\documentclass[a4paper,12pt]{amsart}

\usepackage[headings]{fullpage}

\calclayout

\makeindex

\makeatother

\usepackage{amsfonts,graphics,amsmath,amsthm,amscd,amssymb,latexsym,euscript,enumerate}

\usepackage{epsfig}
\usepackage{flafter}
\usepackage[all,cmtip,line]{xy}
\usepackage{array}
\usepackage[english]{babel}
\usepackage{overpic}
\usepackage{subfig}
\usepackage{multirow}
\usepackage{microtype}
\usepackage{wrapfig}
\usepackage{longtable}
\usepackage{supertabular}
\usepackage{caption}
\usepackage{tikz}
\usetikzlibrary{positioning}
\usepackage{tikz-cd}
\usepackage{float}
\allowdisplaybreaks

\usepackage[dvipsnames,svgnames,table]{xcolor}
\usepackage{graphicx}
\usepackage{hyperref}
\hypersetup{
    colorlinks=true,    
    linkcolor=blue,          
    citecolor=blue,      
    filecolor=blue,      
    urlcolor=blue
}

\newtheorem{theorem}{Theorem}[section]
\newtheorem{lemma}[theorem]{Lemma}

\newtheorem*{theorem*}{Theorem}

\theoremstyle{plain}
\newtheorem{corollary}[theorem]{Corollary}
\newtheorem{problem}[theorem]{Problem}

\theoremstyle{definition} 
\newtheorem{definition}[theorem]{Definition}
\newtheorem{definition-lemma}[theorem]{Definition-Lemma}

\newtheorem{example}[theorem]{Example}

\newtheorem{remark}[theorem]{Remark}
\newtheorem*{notation}{Notation}
\numberwithin{equation}{section}

\newcommand{\C}{\mathbb{C}}
\newcommand{\R}{\mathbb{R}}
\newcommand{\Z}{\mathbb{Z}}

\newcommand{\Q}{\mathbb{Q}}
\newcommand{\OO}{\mathcal{O}}

\newcommand{\msp}{\mathsf{p}}
\def\P{\mathbb{P}}

\newcommand{\A}{\mathbb{A}}
\newcommand{\F}{\mathbb{F}}

\DeclareMathOperator{\Pic}{Pic}

\def\Pic{\operatorname{Pic}}
\def\Proj{\operatorname{Proj}}
\def\Spec{\operatorname{Spec}}

\def\Supp{\operatorname{Supp}}

\def\Sing{\operatorname{Sing}}

\def\gal{\operatorname{Gal}}

\usepackage{mathtools}

\DeclarePairedDelimiterX{\norm}[1]{\lVert}{\rVert}{#1}

\title[Cylindricity of weighted singular del Pezzo surfaces]
{Cylindricity of weighted singular del Pezzo surfaces over fields of characteristic zero}

\begin{document}

\author[I.-K.~Kim]{In-Kyun Kim}
\author[D.-W.~Lee]{Dae-Won Lee}
\author[M.~Sawahara]{Masatomo Sawahara}
\address[In-Kyun Kim]{June E Huh Center for Mathematical Challenges, Korea Institute for Advanced Study, 85 Hoegiro Dongdaemun-gu, Seoul 02455, Republic of Korea.}
\email{soulcraw@kias.re.kr}
\address[Dae-Won Lee]{Department of Mathematics, Ewha Womans University, 52 Ewhayeodae-gil, Seodaemun-gu, Seoul 03760, Republic of Korea}
\email{daewonlee@ewha.ac.kr}
\address[Masatomo Sawahara]{Faculty of Education, Hirosaki University, Bunkyocho 1, Hirosaki-shi, Aomori 036-8560, Japan}
\email{sawahara.masatomo@gmail.com}

\date{\today}

\begin{abstract}
    In this paper, we study the cylindricity of $\Bbbk$-forms of singular del Pezzo surfaces obtained by blowing up weighted projective planes $\P(1,1,m)$ over an arbitrary field $\Bbbk$ of characteristic zero. 
    As an application, we construct vertical cylinders on higher-dimensional fibrations whose generic fibers are such $\Bbbk$-forms.
\end{abstract}

\maketitle



\section{Introduction}\label{sect:intro}

\subsection{Vertical cylinders}

Let $X$ be an algebraic variety defined over a field $\Bbbk$, and let $U$ be a Zariski open subset of $X$. Then $U$ is a {\it cylinder} if it is isomorphic to $\A_{\Bbbk}^1\times Z$ for some affine variety $Z$. Although cylinders are geometrically simple objects and they have received a lot of attention from the viewpoint of unipotent group actions on affine cones over polarized varieties (\cite{KPZ11, KPZ13,KPZ14,CPW16}). Recently, cylinders have also appeared naturally in the birational geometry of Fano varieties (\cite{CPPZ, CPW17}). 

\smallskip

We now recall a general approach to finding cylinders on normal projective varieties. This approach was formalized by Dubouloz and Kishimoto using the minimal model program ({\cite[Corollary 1.3.3]{BCHM}} and {\cite[Lemma 9]{DK19a}}). See {\cite[\S 1]{DK19b}} for details. 
To treat cylinders found in Mori fiber spaces with base varieties of positive dimension, it is useful to introduce the following notion of vertical cylinders:

\begin{definition}[{\cite{DK18}}]
Let $f\colon X \to Y$ be a dominant morphism between normal algebraic varieties defined over $\Bbbk$ and let $U \simeq \A^1_{\Bbbk} \times Z$ be a cylinder on $X$. We say that $U$ is a {\it vertical cylinder} with respect to $f$ if there exists a morphism $g\colon Z \to Y$ such that the restriction of $f$ to $U$ coincides with $g \circ pr_Z\colon U \simeq \A^1_{\Bbbk} \times Z \overset{pr_Z}{\to} Z \overset{g}{\to}Y$.
\end{definition}

Equivalently, the \(\A^1\)-direction is contained in the fibers of \(f\). Such cylinders can be detected on the generic fiber: 

\begin{lemma}[{\cite[Lemma 3]{DK18}}]
Let $f\colon X \to Y$ be a dominant morphism between normal algebraic varieties defined over $\Bbbk$, and let $X_{\eta} = f^{-1}(\eta)$ be the generic fiber of $f$, where $\eta$ is the generic point on $Y$. Then $f$ admits a vertical cylinder if and only if $X_{\eta}$ contains a cylinder defined over the function field $\Bbbk(Y)$ of $Y$. 
\end{lemma}

In the above lemma, when $\dim Y>0$, the existence of a vertical cylinder with respect to $f$ reduces to the study of the generic fiber $X_\eta$; however, the base field of $X_{\eta}$, which is the function field $\Bbbk (Y)$ of $Y$, is not algebraically closed. 
Note that the generic fiber of a Mori fiber space is a Fano variety of Picard rank one. Thus, we obtain the following problem: 

\begin{problem}\label{prob}
Let $V$ be a Fano variety of Picard rank one defined over a field $\Bbbk$ of characteristic zero. When does \(V\) contain a cylinder?
\end{problem}

In the case where $\dim V = 1$, the above problem is not difficult. Indeed, since $V$ is a $\Bbbk$-form of the projective line, $V$ contains a cylinder if and only if it has a $\Bbbk$-rational point. In the case where $\dim V = 2$, note that $V$ is a del Pezzo surface of Picard rank one. When $V$ has at most canonical singularities, Dubouloz-Kishimoto and the third author provided a complete criterion for $V$ to contain cylinders ({\cite{DK18,Saw24}}). This naturally leads to the study of del Pezzo surfaces with at most klt singularities, which are so-called log del Pezzo surfaces. On the other hand, we study Problem \ref{prob} when the assumption ``Picard rank one" is removed, with a view to more effectively finding cylinders of high-dimensional normal projective varieties with fibration structures, which are not necessarily Mori fiber spaces. 
In this paper, we will study the cylindricity of $\Bbbk$-forms of $S_m^n$, which are special kinds of log del Pezzo surfaces. We shall explain $S_m^n$ in the next subsection. 


\subsection{$\Bbbk$-forms of $S_m^n$}

The purpose of this paper is to carry out such a classification for a family of singular del Pezzo surfaces arising from weighted projective planes. Let $\bar{\Bbbk}$ be an algebraic closure of $\Bbbk$, and $S_{m}^{n}$ denote the surface obtained by blowing up $\P(1,1,m)$ at $n$ points in general position on its smooth locus. Its minimal resolution $\widetilde{S}^n_m$ is obtained by blowing up the Hirzebruch surface $\mathbb{F}_m$ of degree $m$ at $n$ points in general position. The surface $S_{m}^{n}$ has a unique quotient singularity of type $\frac{1}{m}(1,1)$. When $m \geq 4$, the anticanonical divisor $-K_{S^n_m}$ is ample for $n \leq m+4$; when $m \in \{2, 3\}$, it remains ample for $n \leq m+5$. These surfaces form a natural class of singular del Pezzo surfaces with non-canonical quotient singularities and unbounded geometric Picard rank. 

\smallskip

Over an algebraically closed field, the geometric properties of $S_m^n$ have been well studied ({\cite{CP20,KL25}}). Moreover, the surface is always rational, and their cylindricity is also partially understood. When $n \le m+3$, {\cite{Saw25}} showed that $S_m^n$ contains an $H$-polar cylinder for every $\Q$-divisor $H$ on $S_m^n$ (see, e.g., {\cite{CPW16}} for definition of a polar cylinder). When $n=m+4$, {\cite{KKW25}} proved that $S_m^n$ does not contain any anticanonical polar cylinder (see also {\cite{KSW26}}). 
Over a field that is not necessarily algebraically closed, rationality and cylindricity become genuinely arithmetic questions. They depend not only on the geometry of the minimal resolution, but also on the Galois action on exceptional curves and on the existence of $\Bbbk$-rational points on distinguished curves. This phenomenon is already visible in Picard rank one. Indeed, \cite{KKW25} shows that $S_{2u-1}^{2u+3}$ does not contain an anticanonical polar cylinder over $\C$. Combined with the minimality constraint on $\Bbbk$-forms of Picard rank one (Corollary \ref{cor:rank1}), this implies that no such $\Bbbk$-form is cylindrical. Our aim is to determine exactly how this arithmetic complexity manifests itself for arbitrary $\Bbbk$-forms of $S_m^n$.

\smallskip

\subsection{Main Results and Applications}

To state the main results, let $S$ be a $\Bbbk$-form of $S_m^n$, $\pi\colon Y\to S$ the minimal resolution, and $Q\subset Y$ the unique $(-m)$-curve. In the cases $n = m+4$ and $n = m+5$, the answer is governed by an additional invariant $\ell_S$ (see Definition~\ref{def:invariant}). By definition, $\ell_S$ is the maximal cardinality of a $\gal(\bar{\Bbbk}/\Bbbk)$-invariant set of pairwise disjoint $(-1)$-curves on $Y_{\bar{\Bbbk}}$ that intersect $Q$ and can be contracted over $\Bbbk$. Roughly speaking, \(\ell_S\) measures the extent to which the geometric \((-1)\)-curves meeting \(Q\) can be simultaneously contracted over the ground field; equivalently, it measures the arithmetic obstruction to reducing \(S\) to a surface on which a cylinder can be constructed. 

The next theorem, which summarizes our first main results on rationality and cylindricity for these \(\Bbbk\)-forms, shows that two different types of arithmetic input govern the problem. For \(1\le n\le m+3\), the relevant arithmetic input is the existence of a \(\Bbbk\)-rational point on \(Q\). For \(n=m+4\) and \(n=m+5\), the decisive invariant is \(\ell_S\).

\begin{theorem}[= Theorems \ref{thm:intermediate}, \ref{thm:m+4} and \ref{thm:m+5}]\label{thm:combined}
    Let \(S\) be a \(\Bbbk\)-form of the surface \(S_{m}^{n}\) for some \(m\geq 2\), and \(Q\) the unique \((-m)\)-curve on the minimal resolution defined over \(\Bbbk\). The rationality and cylindricity of \(S\) are classified as follows:
    \begin{enumerate}[(1)]
        \item {\bf Case \(\mathbf{1\leq n\leq m+3}\).}
        \begin{itemize}
            \item [\textit{(i)}] If \(n\leq m+1\), then \(S\) is always cylindrical, and rational if and only if \(Q(\Bbbk)\neq \emptyset\).
            \item [\textit{(ii)}] If \(n=m+2\) and \(Q(\Bbbk)\neq \emptyset\), then \(S\) is cylindrical and rational.
            \item [\textit{(iii)}] If \(n=m+3\), then \(S\) is always rational, and cylindrical if \(Q(\Bbbk)\neq \emptyset\).
        \end{itemize}
        \item {\bf Case \(\mathbf{n=m+4}\).}
        \begin{itemize}
            \item [\textit{(i)}] If \(\ell_{S}\leq m\), then \(S\) is neither cylindrical nor rational.
            \item [\textit{(ii)}] If \(\ell_{S}=m+1\), then \(S\) is cylindrical and rational.
            \item [\textit{(iii)}] If \(\ell_{S}=m+2\) and \(Q(\Bbbk)\neq \emptyset\), then \(S\) is cylindrical and rational.
            \item [\textit{(iv)}] If \(\ell_{S}=m+4\) and \(m\) is odd, then \(S\) is cylindrical and rational.
        \end{itemize}
        \item {\bf Case \(\mathbf{n=m+5}\).}
       \begin{itemize}
            \item [\textit{(i)}] If \(\ell_{S}\leq m+1\), then \(S\) is neither cylindrical nor rational.
            \item [\textit{(ii)}] If \(\ell_{S}=m+2\) or \(\ell_{S}=m+6\), then \(S\) is cylindrical and rational.
            \item [\textit{(iii)}] If \(\ell_{S}=m+3\) and \(Q(\Bbbk)\neq \emptyset\), then \(S\) is cylindrical and rational.
            \item [\textit{(iv)}] If \(\ell_{S}=m+5\), then \(S\) is always rational, and cylindrical if \(Q(\Bbbk)\neq \emptyset\).
        \end{itemize}
    \end{enumerate}
\end{theorem}


\smallskip

Several noteworthy consequences follow. In particular, when $m$ is odd, Lemma \ref{lem:Fm} implies that $Q(\Bbbk)\neq\emptyset$. Hence, every $\Bbbk$-form of $S_{2u-1}^n$ with $1\le n\le 2u+2$ is both cylindrical and rational (Corollary \ref{cor:interm}). For $\Bbbk$-forms of $S^{2u+3}_{2u-1}$, cylindricity and rationality are equivalent and characterized by $\ell_S \geq 2u$ (Corollary \ref{cor:m+4}). Likewise, for $\Bbbk$-forms of Picard rank one with $Q(\Bbbk)\neq\emptyset$, cylindricity occurs precisely when $n\le m+3$ (Corollary \ref{cor:rank1}), which may be viewed as a weighted analogue of the smooth case treated in \cite{DK18}.

\smallskip

We also study $\Bbbk$-forms of $S_m^{m+4}$. When the base field is the complex number field $\C$, it is known that $S_{2u-1}^{2u+3}$ and $S_{2u}^{2u+4}$ are isomorphic to a weighted complete intersection of two hypersurfaces of degree \(2u\) in \(\P(1,1,u,u,2u-1)\) and a weighted hypersurface of degree \(2u+2\) embedded in \(\P(1,1,u,u+1)\), respectively. We extend this description to arbitrary fields of characteristic zero (see Theorem \ref{thm:embedding}). 

\smallskip

Combining Theorems \ref{thm:combined} and \ref{thm:embedding} with \cite[Lemma 3]{DK18}, we obtain the following: 

\begin{theorem}\label{thm:Fano}
There exists a non-isotrivial Fano threefold \(X\to \mathbb P^1_{\C}\) over $\C$ admitting a vertical cylinder, whose generic fiber is a \(\mathbb C(\mathbb P^1_{\C})\)-form of \(S_{2u-1}^{2u+3}\), where $\C(\mathbb P^1_{\C})$ is the function field of $\mathbb{P}^1_{\C}$. 
\end{theorem}

Thus, our classification serves not only as a birational description of weighted singular del Pezzo surfaces over non-closed fields, but also as a practical tool for constructing cylinders in higher dimensions.

\smallskip

\subsection{Organization of this paper}

The paper is organized as follows. In Section \ref{sect:prelim}, we review basic facts on $\Bbbk$-forms and cylinders. Section \ref{sect:main1} proves the main theorem by reducing each case to one of the model cylinder configurations from Examples \ref{ex:(0)}--\ref{ex:(4)}. 
In Section \ref{sect:main2}, we realize $\Bbbk$-forms of $S_m^{m+4}$ as weighted complete intersections or weighted hypersurfaces. Finally, in Section \ref{sect:examples}, we give explicit examples with prescribed values of $\ell_S$ and apply the classification to construct a non-isotrivial Fano threefold carrying a vertical cylinder.

\section{Preliminaries}\label{sect:prelim}
Throughout this paper, we work over a field \(\Bbbk\) of characteristic zero.


\begin{notation}
We use the following notation: 
\begin{itemize}
\item $\A ^d_{\Bbbk}$: the affine space of dimension $d$ defined over a field $\Bbbk$. 
\item $\P ^d_{\Bbbk}$: the projective space of dimension $d$ defined over a field $\Bbbk$. 
\item $\P(a_0,a_1,\dots,a_d)$: the weighted projective space of weights $a_0,a_1,\dots,a_d$. 
\item $\F _m$: the Hirzebruch surface of degree $m$; i.e., $\F_m = \P(\mathcal{O}_{\P^1}\oplus\mathcal{O}_{\P^1}(-m))$. 
\item $S_m^n$: the surface obtained by blowing up $\P(1,1,m)$ at $n$ points in general position. 
\item $K_X$: the canonical divisor on a variety $X$. 
\item $\rho_{\Bbbk}(X)$: the Picard rank of a variety $X$ defined over $\Bbbk$. 
\item $\Bbbk(X)$: the function field of a variety $X$ over $\Bbbk$. 
\item $D_1\sim D_2$: $D_1$ and $D_2$ are linearly equivalent. 
\item $D_1 \cdot D_2$: the intersection number of $D_1$ and $D_2$. 
\item $D^2$: the self-intersection number of $D$. 
\item $I_{\mathsf{p}}(C,C')$: the intersection multiplicity of $C$ and $C'$ at $\mathsf{p}$. 
\item $\varphi ^{-1}_{\ast}(D)$: the strict transform of $D$ by a  birational morphism $\varphi$. 
\item $\psi _{\ast}(D)$: the direct image of $D$ by a birational morphism $\psi$. 
\item $\Supp (D)$: the support of $D$. 
\item $|D|$: the complete linear system of $D$. 
\item $\bar{\Bbbk}$: an algebraic closure of \(\Bbbk\).
\item $X_{\bar{\Bbbk}}$: the base extension of a variety $X$ defined over $\Bbbk$ to $\bar{\Bbbk}$. 
\item $\varphi_{\bar{\Bbbk}}$: the base extension of a morphism $\varphi$ over $\Bbbk$ to $\bar{\Bbbk}$. 
\item $X(\Bbbk)$: the set of all $\Bbbk$-rational points on a variety $X$. 
\end{itemize}
\end{notation}

Let $\varphi\colon X \to Y$ be a surjective morphism between smooth projective varieties defined over $\Bbbk$. In this article, we say that $\varphi$ is a $\P^1$-fibration and a conic bundle if a general fiber and every fiber of the base extension $\varphi_{\bar{\Bbbk}}\colon X_{\bar{\Bbbk}} \to Y_{\bar{\Bbbk}}$ are isomorphic to $\P^1_{\bar{\Bbbk}}$ and to the plane conic (not necessarily irreducible), respectively. 

We next recall some basic facts about $\Bbbk$-forms, which will be used to relate the geometry of our surfaces over $\bar{\Bbbk}$ to their geometry over $\Bbbk$.

\subsection{$\Bbbk$-forms}\label{subsect:forms}
Let \(\bar{\Bbbk}\) be an algebraic closure of \(\Bbbk\). For a variety $X$ defined over $\bar{\Bbbk}$ and a variety $Y$ defined over $\Bbbk$, we say that $Y$ is a {\it $\Bbbk$-form} of $X$ if the base extension of $Y$ to $\bar{\Bbbk}$ is isomorphic to $X$, i.e., $Y_{\bar{\Bbbk}}\coloneqq Y\times_{\Spec(\Bbbk)} \Spec(\bar{\Bbbk}) \simeq X$. A closed point $\msp\in X$ on a variety $X$ over $\bar{\Bbbk}$ is a {\it $\Bbbk$-rational point} if $\msp$ is defined over $\Bbbk$, in other words, $\msp$ is invariant under the action of ${\rm Gal}(\bar{\Bbbk}/\Bbbk)$. For a variety $X$ over $\Bbbk$, $X(\Bbbk)$ denotes the set of all $\Bbbk$-rational points on $X$. In this subsection, we summarize several results on $\Bbbk$-forms. 

\smallskip

The following result is the Ch\^{a}telet theorem. 

\begin{theorem}[{cf. \cite[Proposition 4.5.10]{Poo17}}]\label{chatelet}
Let $X$ be a $\Bbbk$-form of the $n$-dimensional projective space. Then $X \simeq \P^n_{\Bbbk}$ if and only if $X$ has a $\Bbbk$-rational point. 
\end{theorem}

The following lemmas are well known to experts; however, for the reader's convenience, we provide their proofs.

\begin{lemma}\label{lem:two points}
Let $C$ be a \(\Bbbk\)-form of the projective line. 
Then there exist two points $\msp_1$ and $\msp_2$ on $C_{\bar{\Bbbk}} \simeq \P^1_{\bar{\Bbbk}}$ such that the union $\msp_1+\msp_2$ is defined over $\Bbbk$. 
\end{lemma}

\begin{proof}
Since $C$ is a $\Bbbk$-form of $\P^1$, we know $\OO_C(-K_C) \simeq \OO_C(2)$ and $\deg(-K_C)=2$. Hence, $|-K_C|\not=\emptyset$. 
Take a divisor $\msp \in |-K_C|$ defined over $\Bbbk$. 
Since $\msp_{\bar{\Bbbk}} \in |-K_{C_{\bar{\Bbbk}}}|$, there exist two points $\msp_1$ and $\msp_2$ on $C_{\bar{\Bbbk}} \simeq \P^1_{\bar{\Bbbk}}$ such that $\msp_{\bar{\Bbbk}}=\msp_1+\msp_2$. Since $\msp_{\bar{\Bbbk}}$ is defined over $\Bbbk$, so is $\msp_1+\msp_2$. 
\end{proof}

\begin{lemma}[{cf. \cite[Proposition 1.7]{KSC}}]\label{lem:odd points}
Let $C$ be a \(\Bbbk\)-form of the projective line. 
Assume that $C_{\bar{\Bbbk}}$ has an odd number of points whose union is defined over \(\Bbbk\). Then $C(\Bbbk)\neq\emptyset$. Moreover, $C \simeq \P^1_{\Bbbk}$. 
\end{lemma}

\begin{proof}
Assume that there exist points $\msp_1,\dots,\msp_{2n+1}$ on $C_{\bar{\Bbbk}} \simeq \P^1_{\bar{\Bbbk}}$ such that the union $\msp \coloneqq \msp_1+\dots+\msp_{2n+1}$ is defined over \(\Bbbk\), where $n$ is a non-negative integer. Then the divisor $\msp+nK_{C_{\bar{\Bbbk}}}$ on $C_{\bar{\Bbbk}}$ is defined over \(\Bbbk\), since $-K_{C_{\bar{\Bbbk}}}$ is defined over \(\Bbbk\). Since $\OO_{C_{\bar{\Bbbk}}}(\msp+nK_{C_{\bar{\Bbbk}}}) \simeq \OO_{C_{\bar{\Bbbk}}}(1)$, there exists a point $\msp_0$ on $C_{\bar{\Bbbk}}$ such that $\msp_0 \sim \msp+nK_{C_{\bar{\Bbbk}}}$ so that $\msp_0$ is defined over \(\Bbbk\). This implies that $C$ is a \(\Bbbk\)-form of the projective line with a \(\Bbbk\)-rational point $\msp_0$. Thus, we obtain that $C \simeq \P^1_{\Bbbk}$ by Theorem \ref{chatelet}. 
\end{proof}

\begin{lemma}\label{lem:line conic}
Let $S$ be a \(\Bbbk\)-form of the projective plane. 
If $S_{\bar{\Bbbk}}$ contains a curve $C$ of degree $\le 2$ defined over $\Bbbk$, then $S \simeq \P^2_{\Bbbk}$. 
\end{lemma}

\begin{proof}
Let us first prove the case $\deg(C)=1$. Note that $-K_{S_{\bar{\Bbbk}}}$ is defined over \(\Bbbk\) and $\OO_{S_{\bar{\Bbbk}}}(-K_{S_{\bar{\Bbbk}}}) \simeq \OO_{S_{\bar{\Bbbk}}}(3)$. Hence, there exists a curve $C'$ on $S_{\bar{\Bbbk}}$ such that $C' \sim -K_{S_{\bar{\Bbbk}}}$ so that $C'$ is defined over \(\Bbbk\). We may assume that $C$ and $C'$ intersect exactly three points. By Lemma \ref{lem:odd points}, $C$ has a $\Bbbk$-rational point. Hence, $S \simeq \P^2_{\Bbbk}$ by Theorem \ref{chatelet}. 

Next, let us prove the case $\deg(C)=2$. Since $\OO_{S_{\bar{\Bbbk}}}(-C-K_{S_{\bar{\Bbbk}}}) \simeq \OO_{S_{\bar{\Bbbk}}}(1)$, there exists a curve $C''$ on $S_{\bar{\Bbbk}}$ such that $C'' \sim -C-K_{S_{\bar{\Bbbk}}}$ so that $C''$ is defined over \(\Bbbk\). 
Then $\deg(C'')=1$. Hence, we obtain that $S \simeq \P^2_{\Bbbk}$ by the same argument as above. 
\end{proof}

\begin{lemma}\label{lem:p1p1}
Let $S$ be a $\Bbbk$-form of $\P^1\times\P^1$. If $S$ has a $\Bbbk$-rational point, then $S$ is rational. 
\end{lemma}

\begin{proof}
By assumption, there exists a $\Bbbk$-rational point $\msp$ on $S$. Let $\varphi\colon S'\to S$ be the blow-up at $\msp$, and $E$ the exceptional curve of $\varphi$. Then $S'$ is a smooth del Pezzo surface of degree $7$. Hence, there exist two $(-1)$-curves $E_1$ and $E_2$ on $S'_{\bar{\Bbbk}}$ such that $E_1 \cdot E=E_2 \cdot E=1$ and $E_1 \cdot E_2=0$. Since the disjoint union $E_1+E_2$ is defined over $\Bbbk$, we obtain the contraction $\psi\colon S' \to S''$ of $E_1+E_2$. Then $S''$ is a $\Bbbk$-form of $\P^2$. Moreover, since $\psi_{\ast}(E)$ is of degree one, we obtain $S'' \simeq \P^2_{\Bbbk}$ by Lemma \ref{lem:line conic}. Hence, $S$ is rational.  
\end{proof}

\begin{lemma}\label{lem:Fm}
Let $m$ be a positive integer, and $S_m$ a $\Bbbk$-form of the Hirzebruch surface $\F_m$ of degree $m$. 
Then we have the following: 
\begin{enumerate}
\item If $m$ is odd, then $S_m$ is the Hirzebruch surface of degree $m$, i.e., $S_m \simeq \P_{\Bbbk}(\OO_{\P^1}\oplus\OO_{\P^1}(-m))$ over $\Bbbk$. 
\item When $m$ is even, $S_m$ is the Hirzebruch surface $\F_m$ of degree $m$ if and only if the unique negative section $Q$ on $(S_m)_{\bar{\Bbbk}}$ has a $\Bbbk$-rational point. 
\end{enumerate}
\end{lemma}

\begin{proof}
For (1), we prove this assertion by induction on $m$. First let us assume $m=1$. Note that $S_1$ contains a unique $(-1)$-curve $E$ defined over $\Bbbk$. Hence, we have the contraction $\psi_1\colon S_1 \to S_1'$ of $E$. Then $S_1'$ is a $\Bbbk$-form of $\P^2$ with a $\Bbbk$-rational point $\psi_1(E)$. Thus, $S_1' \simeq \P^2_{\Bbbk}$ by Theorem \ref{chatelet}, and hence, $S_1 \simeq \P_{\Bbbk}(\OO_{\P^1}\oplus\OO_{\P^1}(-1))$. 

From now on, we assume $m>1$. By Lemma \ref{lem:two points}, the minimal section $Q$ contains two points $\msp_1$ and $\msp_2$ over $\bar{\Bbbk}$ such that the union $\msp_1+\msp_2$ is defined over $\Bbbk$. Let $F_1$ and $F_2$ be distinct fibers of the ruling on $S_{\bar{\Bbbk}} \simeq \F_m$ passing through $\msp_1$ and $\msp_2$, respectively. The union $F_1+F_2$ is defined over $\Bbbk$. There exist two points $\msp_1'$ and $\msp_2'$ on $F_1 \setminus \{\msp_1\}$ and $F_2 \setminus \{\msp_2\}$ over $\bar{\Bbbk}$, respectively, such that the union $\msp_1'+\msp_2'$ is defined over $\Bbbk$. Let $\varphi_m\colon S_m' \to S_m$ be the blow-up at $\msp_1+\msp_2$ over $\Bbbk$. Then we obtain the contraction $\psi_m\colon S_m' \to S_{m-2}$ of the strict transform of $F_1+F_2$, defined over $\Bbbk$ so that $S_{m-2}$ is a $\Bbbk$-form of $\F_{m-2}$. By the induction hypothesis, $S_{m-2}\simeq \P_{\Bbbk}(\OO_{\P^1}\oplus\OO_{\P^1}(-(m-2)))$. Therefore, we see that $S_m \simeq \P_{\Bbbk}(\OO_{\P^1}\oplus\OO_{\P^1}(-m))$. 

\smallskip

To prove (2), assume that $Q$ has a $\Bbbk$-rational point \(\msp\). Then the fiber $F$ of the ruling on $(S_m)_{\bar{\Bbbk}}$ passing through $\msp$ is defined over $\Bbbk$. Hence, we obtain the elementary transformation $\phi\colon S_m \dashrightarrow S_{m-1}$ over $\Bbbk$, where $S_{m-1}$ is a $\Bbbk$-form of $\F_{m-1}$. Since $m-1$ is odd, we know that $S_m \simeq \P_{\Bbbk}(\OO_{\P^1}\oplus\OO_{\P^1}(-m))$ by (1). 

Conversely, assume that $S_m \simeq \P_{\Bbbk}(\OO_{\P^1}\oplus\OO_{\P^1}(-m))$. Then $Q$ clearly has a $\Bbbk$-rational point since $Q \simeq \P^1_{\Bbbk}$. 
\end{proof}

\subsection{Cylinders in smooth del Pezzo surfaces}
In this subsection, we recall some results on cylinders in smooth del Pezzo surfaces defined over $\Bbbk$. 

The following theorem gives a strong criterion for rationality and cylindricity.

\begin{theorem}[{cf. \cite[Theorem 1]{DK18}, \cite[Theorem 1.7]{Saw23a}, \cite[Theorem 1.3]{Saw23b}}]\label{thm:Saw cyl iff}
Let $Z$ be a smooth geometrically rational projective surface defined over $\Bbbk$. 
If $Z$ is $\Bbbk$-minimal and $(-K_Z)^2 \le 4$, then $Z$ is neither rational nor cylindrical. 
\end{theorem}

For smooth projective surfaces, cylindricity is a birational invariant.

\begin{theorem}[{cf. \cite[Corollary 1.8]{Saw23b}}]\label{thm:Saw cyl bir}
Let $Y$ and $Z$ be two smooth projective surfaces defined over $\Bbbk$. Assume that $Y$ and $Z$ are birationally equivalent. Then $Y$ is cylindrical if and only if $Z$ is cylindrical. 
\end{theorem}

From now on, we present several examples of cylinders that will help the reader understand the proof of our main result. These configurations will appear repeatedly in Section \ref{sect:main1} when we construct cylinders on various $\Bbbk$-forms.

\begin{example}\label{ex:(0)}
Let $S$ be a $\Bbbk$-form of the Hirzebruch surface $\F_m$ of degree $m>0$. Then the minimal section $M$ of $S_{\bar{\Bbbk}}$ is defined over $\Bbbk$. Let $F_1,\dots,F_r$ be distinct fibers of the ruling on $S_{\bar{\Bbbk}} \simeq \F_m$ such that the union $F_1+\dots+F_r$ is defined over $\Bbbk$. If $\Bbbk$ is algebraically closed, then $U \coloneqq S \setminus \Supp (M+F_1+\dots+F_r)$ is a cylinder. However, $U$ is not always a cylinder if \(\Bbbk\) is not algebraically closed. Nevertheless, $U$ contains a cylinder. 
Let $\Phi\colon S \to C$ be a conic bundle. Note that $\Phi_{\bar{\Bbbk}}\colon \F_m \to \P^1_{\bar{\Bbbk}}$. 
We obtain that the restriction $\varphi \coloneqq  \Phi|_{U}$ gives a morphism over an affine curve $B \subseteq C$. 
By construction, the base extension $\varphi_{\bar{\Bbbk}}$ is an $\A^1$-bundle, and hence, so is $\varphi$ by {\cite[Theorem 1]{KM78}}, which implies that there exists an open subset $Z \subseteq B$ such that $\varphi^{-1}(Z) \simeq \A^1_{\Bbbk} \times Z$. 
\end{example}

\begin{example}\label{ex:(1)}
On $\P^2_{\Bbbk}$, let $\msp$ be a $\Bbbk$-rational point, and $L_1,\dots,L_r$ lines passing through $\msp$, where $r$ is a positive integer. 
Then $\P^2_{\bar{\Bbbk}} \setminus \Supp(L_1+\dots+L_r)$ is a cylinder over $\bar{\Bbbk}$ since it is isomorphic to $\A^1_{\bar{\Bbbk}} \times (\A^1_{\bar{\Bbbk}}\setminus{\text{$(r-1)$ points}})$ (see also {\cite[Example 4.1.1]{CPW17}}). 
Assume further that the union $L_1+\dots +L_r$ is defined over $\Bbbk$. 
Although $U \coloneqq  \P^2_{\Bbbk} \setminus \Supp(L_1+\dots+L_r)$ is defined over $\Bbbk$, it is not always a cylinder. 
On the other hand, there exists a cylinder of $\P^2_{\Bbbk}$ over $\Bbbk$ contained in $U$. 
Let $\varphi\colon \F_1 \to \P^2_{\Bbbk}$ be the blow-up at $\msp$. Then the boundary of $\varphi^{-1}(U) \simeq U$ is the same as that in Example \ref{ex:(0)}. 
\end{example}

\begin{example}\label{ex:(2)}
On $\P^2_{\Bbbk}$, let $\msp$ be a $\Bbbk$-rational point, and let $Q_1,\dots,Q_r$ be irreducible conics on $\P^2_{\bar{\Bbbk}}$ passing through $\msp$ such that $I_{\msp}(Q_i,Q_j) = 4$ for every $i,j=1,\dots ,r$ $(i \not= j)$, where $r$ is a positive integer. 
Note that conics $Q_1,\dots ,Q_r$ have the common tangent line $T$ at $\msp$. 
Assume that the union $Q_1+\dots+Q_r$ is defined over $\Bbbk$ so that $T$ is also defined over $\Bbbk$. 
Then there exists a cylinder of $\P^2_{\Bbbk}$ defined over $\Bbbk$ and contained in $U \coloneqq  \P^2_{\Bbbk} \setminus \Supp(Q_1+\dots+Q_r+T)$. 
Indeed, $U$ can be obtained from $\P^2_{\Bbbk}$ by blowing up $\msp$ and its infinitely near points, and then performing a sequence of contractions starting with the strict transform of $T$, which contains a cylinder. 
\end{example}

\begin{example}\label{ex:(3)}
Let $C$ be a cubic curve defined over $\Bbbk$ with a nodal singular point \(\msp\) on $\P^2_{\Bbbk}$, $\varphi\colon \F_1 \to \P^2_{\Bbbk}$ the blow-up at $\msp$, $E'$ the exceptional curve of $\varphi$, and $C'$ the strict transform of $C$ under $\varphi$. 
Then $C'$ and $E'$ intersect at exactly two points $\msp_1'$ and $\msp_2'$. 
Let $F_1'$ and $F_2'$ be distinct fibers of the ruling on $\F_1$.  
Note that the union $F_1'+F_2'$ is defined over $\Bbbk$ since the union $\msp_1'+\msp_2'$ is defined over $\Bbbk$. 
Then there exists a cylinder of $\F_1$ over $\Bbbk$ contained in $\F_1 \setminus \Supp(C'+F_1'+F_2')$. 
\end{example}

\begin{example}\label{ex:(4)}
Let $S$ be a $\Bbbk$-form of $\P^1 \times \P^1$ with a $\Bbbk$-rational point $\msp$. 
There exist unique irreducible curves $L_1$ and $L_2$ of types $(1,0)$ and $(0,1)$ on $S_{\bar{\Bbbk}} \simeq \P^1_{\bar{\Bbbk}} \times \P^1_{\bar{\Bbbk}}$ passing through $\msp$. 
Note that the union $L_1+L_2$ is defined over $\Bbbk$. 
Let $C_1,\dots,C_r$ be irreducible curves of type $(1,1)$ on $S_{\bar{\Bbbk}} \simeq \P^1_{\bar{\Bbbk}} \times \P^1_{\bar{\Bbbk}}$ passing through $\msp$ such that $I_{\msp}(C_i,C_j) = 2$ for $i,j=1,\dots,r$ $(i\not= j)$, where $r$ is a positive integer. 
Assume that the union $C_1+\dots+C_r$ is defined over $\Bbbk$. 
Then there exists a cylinder of $S$ over $\Bbbk$ contained in $U \coloneqq  S \setminus (C_1+\dots+C_r+L_1+L_2)$. 
Indeed, $U$ can be obtained from $\P^2_{\Bbbk}$ by blowing up at $\msp$ and contracting the strict transform of $L_1+L_2$, which contains a cylinder. 
\end{example}

\section{Proof of main theorem}\label{sect:main1}

In this section, we give a proof of Theorem \ref{thm:combined}.

The cylindricity proofs in this section follow a uniform pattern. 
Starting with the minimal resolution \(Y\), we contract a suitable \(\gal(\bar{\Bbbk}/\Bbbk)\)-invariant configuration of \((-1)\)-curves to obtain a simpler \(\Bbbk\)-form \(W\) of one of the standard surfaces
\(\mathbb F_m\), \(\mathbb P^2\), \(\mathbb P^1\times \mathbb P^1\), or \(\mathbb F_1\). We then choose a boundary on \(W\) fitting one of Examples \ref{ex:(0)}--\ref{ex:(4)} and recover a cylinder on \(S\) from the corresponding open subset of \(W\).

\subsection{Case: $n \le m+3$}
In this subsection, we shall show the following theorem: 

\begin{theorem}\label{thm:intermediate}
Let $S$ be a \(\Bbbk\)-form of $S_m^n$ for $m \ge 2$ and $1 \le n \le m+3$, $\pi\colon Y \to S$ the minimal resolution, and $Q$ the unique $(-m)$-curve on $Y_{\bar{\Bbbk}}$. 
Then the following holds.
\begin{enumerate}
\item If $1\le n \le m+1$, then $S$ is cylindrical. Moreover, $S$ is rational if and only if $Q$ has a \(\Bbbk\)-rational point. 
\item If $n = m+2$ and $Q$ has a \(\Bbbk\)-rational point, then $S$ is both cylindrical and rational. 
\item If $n = m+3$, then $S$ is rational. Moreover, if $Q$ has a \(\Bbbk\)-rational point, then $S$ is cylindrical. 
\end{enumerate}
\end{theorem}

If $m$ is odd, then $Q$ always has a \(\Bbbk\)-rational point by Lemma \ref{lem:Fm}. Hence, we obtain the following corollary.

\begin{corollary}\label{cor:interm}
Let $S$ be a \(\Bbbk\)-form of $S_{2u-1}^{n}$ for $u \ge 2$ and $1 \le n \le 2u+2$. Then $S$ is both cylindrical and rational. 
\end{corollary}

Recall that the minimal resolution \(\tilde{S}_{m}^{n}\) of the surface $S_m^n$ is obtained by blowing up at \(n\) points in general position on the Hirzebruch surface $\F_m$ of degree $m$. Let $F$ be the general fiber of the $\P^1$-fibration $\tilde{S}_{m}^{n} \to \F_m \to \P^1$, $Q$ the unique $(-m)$-curve on \(\tilde{S}\), and $E_1,\dots ,E_n$ be the exceptional curves of the blow-up of $\F_m$. 

\begin{lemma}\label{lem:middle curves}
Let $E$ be a $(-1)$-curve on $\tilde{S}_{m}^{n}$, where $m \ge 2$ and $1 \le n \le m+3$. Then $E$ is linearly equivalent to one of the following.
\begin{enumerate}
\item $E \sim E_i$ for some $i=1,\dots ,n$. 
\item $E \sim F-E_i$ for some $i=1,\dots ,n$. 
\item $E \sim Q + mF - (E_{i_1} + \dots + E_{i_{m+1}})$ for some $1 \le i_1 < i_2 < \dots < i_{m+1} \le n$ if $n \ge m+1$. 
\item $E \sim Q + (m+1)F -(E_1 + \dots + E_{m+3})$ if $n=m+3$. 
\end{enumerate}
\end{lemma}

\begin{proof}
Since $\Pic (\tilde{S}_m^n)$ is generated by $Q$, $F$ and $E_1,\dots ,E_n$, we can write 
\[E \sim aQ + bF - \sum _{i=1}^nc_iE_i\] 
for some $a,b,c_1,\dots ,c_n \in \Z$. 
Since $E \cdot (-K_{\tilde{S}_m^n}) = 1$, we have $1 = (-m+2)a + 2b - \sum _{i=1}^nc_i$. By setting a non-negative integer $d \coloneqq E \cdot Q$, we have $d = -ma + b$. Hence, we obtain $(m+2)a + 2d - 1 = \sum _{i=1}^nc_i$. 

From now on, we may assume $E \not\sim \Delta \coloneqq Q + (m+1)F -(E_1 + \dots + E_n)$ when $n=m+3$. 
Since $|\Delta| \not= \emptyset$, we have $0 \le E \cdot \Delta = a + b - \sum _{i=1}^nc_i = (m+1)a + d - \sum _{i=1}^nc_i$. 
Hence, we obtain that 
\begin{align*}
(m+2)a + 2d - 1 \le (m+1)a + d,
\end{align*}
which implies that \(a + d \le 1\). Note that the number $a$ is a non-negative integer since we have $0 \le E \cdot F = a$. Hence, there are only three possible cases: $(0,0)$, $(1,0)$ and $(0,1)$. 

If $(a,d) = (0,0)$, then $E \sim - \sum _{i=1}^nc_iE_i$. 
Since $E^2 = - \sum _{i=1}^nc_i^2 = -1$, we have $E \sim E_i$ for some $i$. 

If $(a,d) = (0,1)$, then $E \sim F - \sum _{i=1}^nc_iE_i$. 
Since $E^2 = - \sum _{i=1}^nc_i^2 = -1$, we obtain that $E \sim F - E_i$ for some $i$. 

Now, if $(a,d) = (1,0)$, then $E \sim Q + mF - \sum _{i=1}^nc_iE_i$. 
Since the general member of $|Q+mF|$ is smooth rational, \(c_i\) is either \(0\) or \(1\) for every $i$.  By the equality $E^2 = m - \sum _{i=1}^nc_i^2 = -1$, we obtain that $E \sim Q + mF - (E_{i_1} + \dots + E_{i_{m+1}})$ for some indices $1 \le i_1 < i_2 < \dots < i_{m+1} \le n$ if $n \ge m+1$. 

Finally, if $E$ is not linearly equivalent to any of the divisors in (1)--(3), then $E \sim Q + (m+1)F-(E_1+\dots+E_{m+3})$ by the above argument. This completes the proof. 
\end{proof}

\begin{proof}[Proof of Theorem \ref{thm:intermediate}]
Let $F$ be a general fiber of the $\P^1$-fibration $Y_{\bar{\Bbbk}}\coloneqq \tilde{S}_{m}^{n} \to \F_m \to \P^1_{\bar{\Bbbk}}$, $E_1,\dots ,E_n$ the exceptional curves of blowing up from $\F_m$, and $E_i'$ be a $(-1)$-curve on $Y_{\bar{\Bbbk}}$ with $E_i' \sim F-E_i$. 

For (1), the union $\sum_{i=1}^nE_i$ is defined over \(\Bbbk\) by Lemma \ref{lem:middle curves}. Hence, we obtain a contraction $\tau\colon Y \to Z$ of $\sum_{i=1}^nE_i$, defined over \(\Bbbk\). Then $Z$ is a \(\Bbbk\)-form of the Hirzebruch surface $\F_m$ of degree $m$. 
Note that $\hat{Q} \coloneqq \tau_{\ast}(Q)$ and $\sum_{i=1}^n\hat{E}_i' \coloneqq \tau_{\ast}\left(\sum_{i=1}^nE_i'\right)$ are defined over $\Bbbk$. 
Since $Z \setminus \Supp \Bigg( \hat{Q} +  \sum_{i=1}^n\hat{E}_i'\Bigg)$ contains a cylinder of $Z$ (see Example \ref{ex:(0)}), there exists a cylinder of $S$ contained in
\begin{align*}
S \setminus \Supp \Bigg(\pi_{\ast}\bigg(\sum_{i=1}^n(E_i+E_i')\bigg)\Bigg) 
&\simeq Y \setminus \Supp \bigg(Q + \sum_{i=1}^n(E_i+E_i')\bigg) \\
&\simeq Z \setminus \Supp \Bigg( \hat{Q} +  \sum_{i=1}^n\hat{E}_i'\Bigg).
\end{align*}
In this case, therefore, $S$ is cylindrical. 

If $Q$ has a \(\Bbbk\)-rational point, then the image under $\tau$ is also \(\Bbbk\)-rational. By Lemma \ref{lem:Fm}, $Z$ is a trivial \(\Bbbk\)-form of $\F_m$, which implies $Z \simeq \F_m$. In particular, $Z$ is rational and so is $S$. 
Conversely, if $S$ is rational, then there exists a general fiber $F$ defined over \(\Bbbk\). Then the intersection point of $F$ and $Q$ is a \(\Bbbk\)-rational point. This completes the proof of \((1)\).

\smallskip

Now, we give a proof of \((2)\). We note that the union $\sum_{i=1}^{m+2}E_i'$ is defined over \(\Bbbk\) by Lemma \ref{lem:middle curves}. Hence, we obtain a contraction $\tau\colon Y \to Z$ of $\sum_{i=1}^{m+2}E_i'$, defined over \(\Bbbk\). Then $Z$ is a \(\Bbbk\)-form of $\P^1 \times \P^1$ and $\hat{Q} \coloneqq \tau _{\ast}(Q)$ is a \(\Bbbk\)-form of an irreducible curve of type $(1,1)$. 
By the assumption, $\hat{Q}$ contains a \(\Bbbk\)-rational point $\msp$. 
Hence, $S$ is rational by Lemma \ref{lem:p1p1}. 
Moreover, there exist two $0$-curves $\hat{F}$ and $\hat{F}'$ on $Z_{\bar{\Bbbk}}$ passing through $\msp$, which are \(\Bbbk\)-forms of irreducible curves of types $(1,0)$ and $(0,1)$, respectively, such that $\hat{F}+\hat{F}'$ is defined over \(\Bbbk\). 
Then there exists a cylinder of $S$ contained in 
\begin{align*}
S \setminus \Supp \Bigg(\pi_{\ast}\bigg( \sum_{i=1}^{m+2}E_i' + F+F'\bigg)\Bigg) 
&\simeq Y \setminus \Supp \Bigg( Q + \sum_{i=1}^{m+2}E_i' + F+F'\Bigg) \\
&\simeq Z \setminus \Supp (\hat{Q} + \hat{F}+\hat{F}'),
\end{align*}
where $F+F' \coloneqq \tau^{-1}_{\ast}(\hat{F}+\hat{F}')$ (see Example \ref{ex:(4)}). 
Hence, $S$ is cylindrical. 

\smallskip

To prove (3), we first note that there exists a unique $(-1)$-curve $E'$ such that $$E' \sim Q+(m+1)F-\sum_{i=1}^{m+3}E_i$$ by Lemma \ref{lem:middle curves}. Then the only $(-1)$-curves on $Y_{\bar{\Bbbk}}$ meeting $Q$ are only $E_1',\dots,E_{m+3}'$ and $E'$. This implies that the union $E'+\sum_{i=1}^{m+3}E_i'$ is defined over \(\Bbbk\). Hence, we obtain a contraction $\tau\colon Y \to Z$ of $E'+\sum_{i=1}^{m+3}E_i'$, defined over \(\Bbbk\). Thus, $Z$ is a \(\Bbbk\)-form of $\P^2_{\bar{\Bbbk}}$. More precisely, since $Z_{\bar{\Bbbk}}$ contains a conic $\hat{Q} \coloneqq \tau_{\ast}(Q)$ defined over \(\Bbbk\), $Z \simeq \P^2_{\Bbbk}$ by Lemma \ref{lem:line conic}. Hence, $S$ is rational. Moreover, by assumption, $\hat{Q}$ has a \(\Bbbk\)-rational point $\msp$. Let $\hat{T}$ be a tangent line of $\hat{Q}$ at $\msp$. Note that $\hat{T}$ is defined over \(\Bbbk\). 
Then there exists a cylinder of $S$ contained in 
\begin{align*}
S \setminus \Supp \Bigg(\pi_{\ast}\bigg( E' + \sum_{i=1}^{m+3}E_i' + T\bigg)\Bigg) 
&\simeq Y \setminus \Supp \Bigg( Q + E'+ \sum_{i=1}^{m+3}E_i' + T\Bigg) \\
&\simeq Z \setminus \Supp (\hat{Q} + \hat{T}),
\end{align*}
where $T \coloneqq \tau^{-1}_{\ast}(\hat{T})$ (see Example \ref{ex:(2)}). 
Therefore, $S$ is cylindrical. 
\end{proof}

\subsection{The invariant $\ell_S$}

In the rest of this section, we focus on the case when \(S\) is a \(\Bbbk\)-form of \(S_{m}^{m+4}\) or \(S_{m}^{m+5}\). To characterize the cylindricity and rationality in these cases, we need to define the invariant \(\ell_{S}\).

\begin{definition} \label{def:invariant}
Let \(S\) be a \(\Bbbk\)-form of either \(S_{m}^{m+4}\) or \(S_{m}^{m+5}\), \(\pi\colon Y\to S\) the minimal resolution, and \(Q\) the exceptional curve of \(\pi\), which is a $(-m)$-curve.

Let $\mathcal{E}$ be the collection of finite sets $\Sigma$ of $(-1)$-curves on $Y_{\bar{\Bbbk}}$ satisfying the following conditions:
\begin{enumerate}
    \item The curves in $\Sigma$ are pairwise disjoint, and each curve in \(\Sigma\) intersects $Q$.
    \item The set $\Sigma$ is invariant under the Galois action of $\gal(\bar{\Bbbk}/\Bbbk)$. 
\end{enumerate}
    We define the invariant $\ell_{S}$ as \(\ell_S \coloneqq \max \{ |\Sigma| \mid \Sigma \in \mathcal{E} \}.\)
\end{definition}

For the reader’s convenience, we shall calculate the invariant $\ell_S$ for an explicit log del Pezzo surface $S$. 

\begin{example}
Consider the surface
\[
S\coloneqq \left\{w^2=x^4+y^4+(x^2+y^2)z^2\right\}\subset \P(1,1,1,2)_{[x:y:z:w]},
\]
defined over \(\Q\). This is a \(\Q\)-form of \(S_2^6\), and \(S_{\bar{\Q}}\) has a unique singular point
\[
\msp=[0:0:1:0]
\]
of type \(\mathrm{A}_1\). A direct computation shows that \(S\) contains exactly \(12\) lines:
\begin{align*}
&L_1=\{y=ix,\ w=\sqrt{2}\,x^2\}, &
&L_2=\{y=ix,\ w=-\sqrt{2}\,x^2\}, \\
&L_3=\{y=-ix,\ w=\sqrt{2}\,x^2\}, &
&L_4=\{y=-ix,\ w=-\sqrt{2}\,x^2\}, \\
&L_5=\left\{y=e^{\frac{\pi i}{4}}x,\ w=\sqrt[4]{2}\,e^{\frac{\pi i}{8}}xz\right\}, &
&L_6=\left\{y=e^{\frac{\pi i}{4}}x,\ w=-\sqrt[4]{2}\,e^{\frac{\pi i}{8}}xz\right\}, \\
&L_7=\left\{y=-e^{\frac{\pi i}{4}}x,\ w=\sqrt[4]{2}\,e^{\frac{\pi i}{8}}xz\right\}, &
&L_8=\left\{y=-e^{\frac{\pi i}{4}}x,\ w=-\sqrt[4]{2}\,e^{\frac{\pi i}{8}}xz\right\}, \\
&L_9=\left\{y=e^{\frac{3\pi i}{4}}x,\ w=\sqrt[4]{2}\,e^{\frac{7\pi i}{8}}xz\right\}, &
&L_{10}=\left\{y=e^{\frac{3\pi i}{4}}x,\ w=-\sqrt[4]{2}\,e^{\frac{7\pi i}{8}}xz\right\}, \\
&L_{11}=\left\{y=-e^{\frac{3\pi i}{4}}x,\ w=\sqrt[4]{2}\,e^{\frac{7\pi i}{8}}xz\right\}, &
&L_{12}=\left\{y=-e^{\frac{3\pi i}{4}}x,\ w=-\sqrt[4]{2}\,e^{\frac{7\pi i}{8}}xz\right\}.
\end{align*}
Let \(\pi\colon Y\to S\) be the minimal resolution and set $\tilde{L}_i \coloneqq  \pi^{-1}_{\ast}(L_i)$ for $i=1,\dots ,12$. 
Note that the divisor $\sum_{j=1}^{12}\tilde{L}_j$ has the configuration shown in Figure \ref{fig:L}. 
Since the union $\sum_{j=1}^{12}L_j$ decomposes into two ${\rm Gal}(\bar{\Q}/\Q)$-orbits: 
\begin{align*}
L_1 + \dots + L_4,\quad L_5 + \dots + L_{12},
\end{align*}
any ${\rm Gal}(\overline{\Q}/\Q)$-orbit of $\tilde{L}_j$ cannot be contracted over $\Q$. Hence, \(\ell_S=0\). 
On the other hand, over \(\R\) the same surface satisfies \(\ell_S=6\). 
Indeed, since the union $\sum_{j=1}^{12}\tilde{L}_j$ decomposes into six ${\rm Gal}(\C/\R)$-orbits: 
\begin{align*}
L_1 + L_3,\quad L_2 + L_4,\quad L_5 + L_{11},\quad L_6 + L_{12},\quad L_7 + L_9,\quad L_8 + L_{10},
\end{align*}
we obtain the contraction $\tau: Y \to \F_1$ of $\Sigma$ defined over \(\R\), where $\Sigma \coloneqq  \tilde{L}_1+\tilde{L}_3+\tilde{L}_5+\tilde{L}_6+\tilde{L}_{11}+\tilde{L}_{12}$. 

If the base field is $\Q$ (resp. $\R$), then $S$ is neither cylindrical nor rational (resp. both cylindrical and rational) by virtue of the invariant $\ell_S$ combined with Theorem \ref{thm:m+4} below. 
\begin{figure}[ht]
\begin{center}
\begin{tikzpicture}[scale=0.35]
\draw (2,-1) -- (1,3);
\draw (2,6) -- (1,2);
\draw (4,-1) -- (3,3);
\draw (4,6) -- (3,2);
\draw (6,-1) -- (5,3);
\draw (6,6) -- (5,2);
\draw (8,-1) -- (7,3);
\draw (8,6) -- (7,2);
\draw (10,-1) -- (9,3);
\draw (10,6) -- (9,2);
\draw (12,-1) -- (11,3);
\draw (12,6) -- (11,2);

\node at (2,-1.5) {$\tilde{L}_1$};
\node at (4,-1.5) {$\tilde{L}_3$};
\node at (6,-1.5) {$\tilde{L}_5$};
\node at (8,-1.5) {$\tilde{L}_6$};
\node at (10,-1.5) {$\tilde{L}_9$};
\node at (12,-1.5) {$\tilde{L}_{10}$};
\node at (2,6.5) {$\tilde{L}_2$};
\node at (4,6.5) {$\tilde{L}_4$};
\node at (6,6.5) {$\tilde{L}_8$};
\node at (8,6.5) {$\tilde{L}_7$};
\node at (10,6.5) {$\tilde{L}_{12}$};
\node at (12,6.5) {$\tilde{L}_{11}$};
\end{tikzpicture}
\caption{Configuration of $\sum_{j=1}^{12}\tilde{L}_j$}\label{fig:L}
\end{center}
\end{figure}
\end{example}

\subsection{Case: $n=m+4$}
In this subsection, we deal with the case when \(S\) is a \(\Bbbk\)-form of \(S_{m}^{m+4}\).

\begin{theorem}\label{thm:m+4}
Let \(S\) be a \(\Bbbk\)-form of \(S_{m}^{m+4}\). The following assertions hold.
\begin{enumerate}[(1)]
\item If $\ell _S \le m$, then $S$ is neither cylindrical nor rational. 
\item If $\ell _S > m$, then $\ell _S$ is one of $m+1$, $m+2$, $m+4$. 
\item If $\ell _S = m+1$, then $S$ is both cylindrical and rational. 
\item If $\ell _S = m+2$ and $Q(\Bbbk) \neq\emptyset$, then $S$ is both cylindrical and rational. 
\item If $\ell _S = m+4$ and $m=2u-1$, then $S$ is both cylindrical and rational. 
\end{enumerate}
\end{theorem}

By Theorem \ref{thm:m+4} and Lemma \ref{lem:odd points}, we obtain the following corollary.

\begin{corollary}\label{cor:m+4}
Let $S$ be a \(\Bbbk\)-form of $S_{2u-1}^{2u+3}$ for \(u\geq 2\). 
Then
\begin{align*}
\text{$S$ is cylindrical} \iff \text{$S$ is rational} \iff \ell_{S} \ge 2u. 
\end{align*}
\end{corollary}

In what follows, we shall prove Theorem \ref{thm:m+4}. 
By {\cite[Lemma 2.14]{CP20}}, there exists a birational morphism $\eta \colon \widetilde{S}_m^{m+4} \to \P^2_{\bar{\Bbbk}}$, which is a blow-up at $(m+4)$ points on an irreducible conic \(C\) on \(\P^2\) and at one point outside of \(C\) (see also {\cite[\S 4]{KKW25}}). 
Using this birational model, we prove the following description of the $(-1)$-curves disjoint from $Q$: 

\begin{lemma}\label{lem:(-1)curves}
Let $\eta \colon \widetilde{S}_m^{m+4} \to \P^2$ be a blow-up at $(m+4)$ points on an irreducible conic \(C\) on \(\P^2\) and at one point outside of \(C\). Let $e_0$ be the strict transform of a general line on $\P^2$ and $e_1,\dots,e_{m+5}$ the exceptional curves of $\eta$. Then there exist \((2m+8)\) \((-1)\)-curves $E_1,\dots ,E_{m+4}, E_1',\dots, E_{m+4}'$ such that $E_i \sim e_i$ and $E_i' \sim e_0-e_i-e_{m+5}$ for $i=1,\dots ,m+4$. 
Let $E$ be a $(-1)$-curve on $\widetilde{S} _m^{m+4}$ that does not meet $Q\coloneqq \eta_{\ast}^{-1}C$. 
Then we have the following.
\begin{enumerate}[(1)]
\item There exists $0 \le d \le \lfloor \frac{m}{2}+2 \rfloor$ such that 
\begin{align*}
E \sim de_0 - e_{i_1} - e_{i_2} - \dots - e_{i_{2d}} - (d-1)e_{m+5}
\qquad (1 \le i_1 < i_2 < \dots < i_{2d} \le m+4). 
\end{align*}
Here, if $m = 2n-1$ (resp. $m=2n$), then $\lfloor \frac{m}{2}+2 \rfloor = n+1$ (resp. $\lfloor \frac{m}{2}+2 \rfloor = n+2$). 
\item $E \cdot (E_i + E_i') = 1$ for every $i=1,\dots,m+4$. 
\end{enumerate}
\end{lemma}

\begin{proof}
Let $E \sim de_0 - \sum _{j=1}^{m+5}\mu _j e_j$ for some non-negative integers $d$ and $\mu_1,\dots ,\mu _{m+5}$. Since $Q \sim 2e_0 - \sum _{j=1}^{m+4}e_j$ and $E$ is a $(-1)$-curve, we have the following equalities
\begin{align*}
0 &= E \cdot Q = 2d - \sum _{j=1}^{m+4} \mu _j, \\
-1 &= E^2 = d^2 - \sum _{j=1}^{m+4} \mu _j^2 - \mu _{m+5}^2, \\
1 &= E \cdot (-K_{\widetilde{S}_m^{m+4}}) = 3d - \sum _{j=1}^{m+4} \mu _j - \mu _{m+5}. 
\end{align*}
Hence, we obtain $\mu _{m+5} = d-1$. In particular, we obtain that \(-1 = d^2 - \sum _{j=1}^{m+4} \mu _j^2 - (d-1)^2\) which implies that \(0 = 2d - \sum _{j=1}^{m+4} \mu _j^2\).
Thus, we have \(2d = \sum _{j=1}^{m+4} \mu _j = \sum _{j=1}^{m+4} \mu _j^2,\)
which implies that $0 \le d \le \lfloor \frac{m}{2}+2 \rfloor$ and
\begin{align*}
\sum _{j=1}^{m+4}\mu _j e_j = e_{i_1} + e_{i_2} + \dots + e_{i_{2d}}\qquad (1 \le i_1 < i_2 < \dots < i_{2d} \le m+4). 
\end{align*}
This completes the proof. 
\end{proof}

\begin{remark}
By Lemma \ref{lem:(-1)curves}, the configuration of $Q+E_1+\dots+E_{m+4}+E_1'+\dots+E_{m+4}'$ looks like that in Figure \ref{fig:m+4}. 
\begin{figure}[ht]
\begin{center}
\begin{tikzpicture}[scale=0.35]
\draw [thick] (0,5) 
.. controls (0,5) .. (7,5)
.. controls (7,5) and (8,4.75) .. (8.25,2.5)
.. controls (8.25,2.5) and (8,0.25) .. (7,0.5)
.. controls (7,0.5) .. (0,0.5);
\draw (2,-1) -- (1,3);
\draw (2,6) -- (1,2);
\draw (6,-1) -- (5,3);
\draw (6,6) -- (5,2);
\node at (3.75,3.5) {$\cdots$};
\node at (3.75,1.5) {$\cdots$};

\node at (-1,5) {$Q$};
\node at (2,-1.5) {$E_1$};
\node at (6,-1.5) {$E_{m+4}$};
\node at (2,6.5) {$E_1'$};
\node at (6,6.5) {$E_{m+4}'$};
\end{tikzpicture}
\caption{Configuration of $Q+E_1+\dots+E_{m+4}+E_1'+\dots+E_{m+4}'$}\label{fig:m+4}
\end{center}
\end{figure}
\end{remark}

From now on, we prove Theorem \ref{thm:m+4}, which describes how the invariant $\ell_S$
restricts the possible configurations of $(-1)$‑curves and determines rationality and cylindricity.

Let $S$ be a \(\Bbbk\)-form of $S_m^{m+4}$, and $\pi\colon Y \to S$ the minimal resolution. 
Note that $Y_{\bar{\Bbbk}} \simeq \widetilde{S} _m^{m+4}$ contains a $(-m)$-curve $Q$, which is defined over \(\Bbbk\), and there exist $(2m+8)$ $(-1)$-curves $E_1,\dots ,E_{m+4}, E_1',\dots E_{m+4}'$ meeting $Q$ such that $E_i \cdot E_j' = \delta _{i,j}$. 

Then a divisor $E_1+E_1'$ defines a $\P ^1$-fibration $\Phi \colon Y_{\bar{\Bbbk}} \to \P ^1_{\bar{\Bbbk}}$, which has exactly $m+4$ singular fibers $F_i \coloneqq E_i+E_i'$ $(i=1,\dots ,m+4)$. Note that $\Phi$ is defined over \(\Bbbk\) since $\sum _{j=1}^{m+4}F_i$ is defined over \(\Bbbk\). In other words, we obtain a surjective morphism $\varphi \colon Y \to B$ over \(\Bbbk\) such that $\varphi _{\bar{\Bbbk}} = \Phi$. We may assume that $E_1+\dots +E_{\ell _S}$ can be contracted over \(\Bbbk\). Thus, we obtain the contraction $\tau \colon Y \to Z$ of $E_1+\dots +E_{\ell_S}$, defined over \(\Bbbk\). Then there exists a surjective morphism $\psi \colon Z \to B$ defined over \(\Bbbk\) such that $\psi \circ \tau = \varphi$. 

\begin{align*}
\xymatrix{
& Y \ar[dl]_{\pi} \ar[d]^{\varphi} \ar[rrr]^{\tau} &&& Z \ar[d]^{\psi} \\
S & B \ar@{=}[rrr] &&& B
}
\end{align*}

\begin{lemma}\label{lem:not rational}
If $\ell_S \le m$, then $Z$ is \(\Bbbk\)-minimal. 
\end{lemma}
\begin{proof}
Suppose, on the contrary, that $Z$ is not \(\Bbbk\)-minimal. In other words, there exists a smoothly contractible curve $\bar{E}$ on $Z$. 
Let $r$ be the number of such smoothly contractible curves $\bar{E}_{\bar{\Bbbk}}$. By construction of $\tau$, note that every irreducible component of $\bar{E}_{\bar{\Bbbk}}$ is not a fiber component of $\psi_{\bar{\Bbbk}}$. Hence, every irreducible component of $\bar{E}_{\bar{\Bbbk}}$ is a $(-1)$-curve and a one-section of $\psi _{\bar{\Bbbk}}$. That is, $-\bar{E}^2 = \bar{E} \cdot (-K_Z) = r$ and $r(-K_Z \cdot \bar{F}) = 2\bar{E} \cdot \bar{F}$, where $\bar{F}$ is the Galois orbit of a general fiber of $\psi_{\bar{\Bbbk}}$. We consider the divisor $\bar{\Delta} \coloneqq -rK_Z - 2\bar{E}$ on $Z$. Then $\bar{\Delta} \cdot \bar{F} = 0$ and $\bar{\Delta}^2 = -(m-\ell_S)r^2 - 4r < 0$ by the assumption $\ell _S \le m$. Hence, by the Riemann--Roch theorem and since $Z$ is geometrically rational, there exists an integer $\lambda \gg 0$ such that
\begin{align*}
\dim |\bar{\Delta} + \lambda \bar{F}| \ge \frac{1}{2}(\bar{\Delta} + \lambda \bar{F}) \cdot (\bar{\Delta} + \lambda \bar{F} - K_{Z}) 
= \frac{1}{2}\bar{\Delta} \cdot (\bar{\Delta} - K_{Z}) + \frac{\lambda}{2}\bar{F} \cdot (- K_{Z}) \ge 0. 
\end{align*}
Hence, there exists an effective $\Z$-divisor $\bar{D}$ on $Z$ defined over \(\Bbbk\) such that $\bar{D} \sim \bar{\Delta} + \lambda \bar{F}$. Since $\bar{D} \cdot \bar{F} = 0$, every irreducible component of $\bar{D}$ is a fiber component of $\psi$. Moreover, from $\bar{D}^2 = \bar{\Delta}^2 < 0$, we deduce that there exists a smoothly contractible curve over \(\Bbbk\), which is contained in a fiber component of $\psi$. This contradicts the construction of $\tau$. 
\end{proof}

\begin{proof}[Proof of Theorem \ref{thm:m+4}]
We first prove \((1)\). Assume $\ell_S \le m$. By Lemma \ref{lem:not rational}, $Z$ is \(\Bbbk\)-minimal. Moreover, $\rho _{\Bbbk}(Z) = 2$ since \(Z\) admits a fibration $\psi$. 
Since $(-K_Z)^2 = -m + 4 + \ell _S \le 4$, we know that $Z$ is not cylindrical by Theorem \ref{thm:Saw cyl iff}. 
Hence, $Y$ is neither cylindrical nor rational by Theorem \ref{thm:Saw cyl bir}, which implies \(S\) is also neither cylindrical nor rational.

\smallskip

From now on, let us assume that $\ell_S > m$, which implies that $Z$ is a smooth del Pezzo surface of degree $\ge 5$. Since $\rho_{\Bbbk}(Z) > 1$, we know that $(-K_Z)^2 \in \{ 5,6,7,8\}$. Thus, we have $m+1 \le \ell _S \le m+4$. 

Suppose that $\ell_S = m+3$. 
Then $Z$ is a smooth del Pezzo surface of degree $7$. 
Hence, there exists a unique $(-1)$-curve $\bar{L}$ on $Z_{\bar{\Bbbk}}$ such that $\bar{L} \cdot (\bar{E}_{m+4}+\bar{E'}_{m+4}) = 1$, where $\bar{E}_{m+4} \coloneqq \tau _{\ast}(E_{m+4})$ and $\bar{E'}_{m+4} \coloneqq \tau _{\ast}(E_{m+4}')$. 
Note that $Z_{\bar{\Bbbk}}$ contains exactly three $(-1)$-curves $\bar{E}_{m+4}$ and $\bar{E'}_{m+4}$ and $\bar{L}$. 
By the configuration of $\bar{E}_{m+4}+\bar{E'}_{m+4}+\bar{L}$, the curves $\bar{E}_{m+4}$ and $\bar{E'}_{m+4}$ are not contained in the same Galois orbit, which is a contradiction to $\ell_S = m+3$. Therefore, this proves \((2)\).

\smallskip

Now, let us give a proof of \((3)\). Assume that $\ell_S = m+1$. Then $Z$ is a smooth del Pezzo surface of degree $5$. Hence, there exist exactly four $(-1)$-curves $\bar{E}_1'',\bar{E}_2'',\bar{E}_3''$ and $\bar{E}_4''$, which are not fiber components of $\psi _{\bar{\Bbbk}}$. Note that $\sum_{j=1}^4\bar{E}_j''$ is defined over \(\Bbbk\) and is a disjoint union. Hence, so is $\sum_{j=1}^4E_j'' \coloneqq \tau ^{-1}_{\ast}\bigg(\sum_{j=1}^4\bar{E}_j''\bigg)$. 
Let $\sigma \colon Z \to W$ be the contraction of $\sum_{j=1}^4\bar{E}_j''$, defined over \(\Bbbk\). Then $W$ is a \(\Bbbk\)-form of $\P ^2$. For simplicity, set $\hat{Q} \coloneqq (\sigma \circ \tau)_{\ast}(Q)$. Since $\hat{Q}$ is a $1$-curve, we know $\hat{Q} \simeq \P ^1_{\Bbbk}$ and $W \simeq \P ^2_{\Bbbk}$ by Lemma \ref{lem:line conic}. That is, $S$ is rational. Set $\msp_i \coloneqq \sigma _{\bar{\Bbbk}}(\bar{E}''_i)$ for $i=1,2,3,4$. Since $\hat{Q} \simeq \P ^1_{\Bbbk}$, we can take a general \(\Bbbk\)-rational point $\msp_0$ on $\hat{Q}$. For $i=1,2,3,4$, let $\hat{L}_i$ be the line on $W_{\bar{\Bbbk}} \simeq \P ^2_{\bar{\Bbbk}}$ passing through $\msp_0$ and $\msp_i$. Note that $\sum_{j=1}^4\hat{L}_j$ is defined over \(\Bbbk\). Hence, so is $\sum_{j=1}^4L_j \coloneqq (\sigma \circ \tau)^{-1}_{\ast}\left(\sum_{j=1}^4\hat{L}_j\right)$. 
Then there exists a cylinder of $S$ contained in 
\begin{align*}
S \setminus \Supp \Bigg(\pi _{\ast}\bigg(\sum_{i=1}^{m+1}E_i + \sum_{j=1}^4(E_j''+L_j)\bigg)\Bigg) 
&\simeq Y \setminus \Supp \bigg(Q + \sum_{i=1}^{m+1}E_i + \sum_{j=1}^4(E_j''+L_j)\bigg) \\
& \simeq W \setminus \Supp \bigg(\hat{Q}+\sum_{j=1}^4\hat{L}_j\bigg),
\end{align*}
which implies that $S$ is cylindrical (see Example \ref{ex:(1)}). 

Next, let us assume that $\ell_S = m+2$ and $Q(\Bbbk) \neq\emptyset$. Then $Z$ is a smooth del Pezzo surface of degree $6$. Hence, there exist exactly two $(-1)$-curves $\bar{E}_1''$ and $\bar{E}_2''$ on $Z_{\bar{\Bbbk}}$, which are not fiber components of $\psi _{\bar{\Bbbk}}$, such that the union $\bar{E}_1''+\bar{E}_2''$ is disjoint. Hence, we obtain the contraction $\sigma \colon Z \to W$ of $\bar{E}_1''+\bar{E}_2''$, defined over \(\Bbbk\). Then $W$ is a \(\Bbbk\)-form of $\P^1 \times \P^1$. 
Since $Q(\Bbbk) \neq\emptyset$, we know that $\hat{Q}(\Bbbk) \neq\emptyset$, where $\hat{Q} \coloneqq (\sigma \circ \tau)_{\ast}(Q)$. Hence, $S$ is rational by Lemma \ref{lem:p1p1}. 

Since $\hat{Q}(\Bbbk) \not= \emptyset$, we can take a $\Bbbk$-rational point $\msp_0$ on $\hat{Q}$. Note that $\hat{Q}_{\bar{\Bbbk}}$ is of type $(1,1)$. 
Let $\msp_i \coloneqq \sigma (\bar{E}_i'')$ for $i=1,2$, and $\hat{L}_1$ and $\hat{L}_2$ be irreducible curves of types $(1,0)$ and $(0,1)$ on $W_{\bar{\Bbbk}}$ passing through $\msp_0$, respectively. Note that $\hat{L}_1+\hat{L}_2$ is defined over \(\Bbbk\). Let $E_1''+E_2'' \coloneqq \tau ^{-1}_{\ast}(\bar{E}_1''+\bar{E}_2'')$ and $L_1+L_2 \coloneqq (\sigma \circ \tau)^{-1}_{\ast}(\hat{L}_1+\hat{L}_2)$. Assume that $\msp_1,\msp_2 \in \Supp (\hat{L}_1+\hat{L}_2)$. 
Then we have that there exists a cylinder of $S$ contained in 
\begin{align*}
&S \setminus \Supp \Bigg(\pi_{\ast}\bigg(\sum_{i=1}^{m+2}E_i + E_1'' + E_2'' + L_1 + L_2\bigg)\Bigg) \\
&\qquad \qquad \simeq Y \setminus \Supp \bigg(Q + \sum_{i=1}^{m+2}E_i + E_1'' + E_2'' + L_1 + L_2\bigg) \\
&\qquad \qquad \simeq W \setminus \Supp (\hat{Q} + \hat{L}_1 + \hat{L}_2), 
\end{align*}
which implies that $S$ is cylindrical (see Example \ref{ex:(4)}). 

From now on, let us assume that $\msp_1,\msp_2 \not\in \Supp (\hat{L}_1+\hat{L}_2)$. 
Let $\hat{C}_1$ be an irreducible curve of type $(1,1)$ on $W_{\bar{\Bbbk}}$ passing through $\msp_1$ and satisfying $I_{\msp_0}(\hat{Q},\hat{C}_i) = 2$. If $\msp_2 \in \hat{C}_1$, then $\hat{C}_1$ is defined over \(\Bbbk\). 
Set $C_1 \coloneqq (\sigma \circ \tau)^{-1}_{\ast}(\hat{C}_1)$. 
Then there exists a cylinder of $S$ contained in 
\begin{align*}
&S \setminus \Supp \Bigg(\pi_{\ast}\bigg(\sum_{i=1}^{m+2}E_i + E_1'' + E_2'' + L_1 + L_2 + C_1\bigg)\Bigg) \\
&\qquad \qquad \simeq Y \setminus \Supp \bigg(Q+\sum_{i=1}^{m+2}E_i + E_1'' + E_2'' + L_1 + L_2 + C_1\bigg) \\
&\qquad \qquad \simeq W \setminus \Supp (\hat{Q}+\hat{L}_1+\hat{L}_2+\hat{C}_1),
\end{align*}
which implies that $S$ is cylindrical (see Example \ref{ex:(4)}). 

In what follows, let us assume that $\msp_2 \not\in \hat{C}_1$. Then there exists a unique irreducible curve $\hat{C}_2$ on $W_{\bar{\Bbbk}}$ such that ${\rm Gal}(\bar{\Bbbk}/\Bbbk) \cdot \hat{C}_1 = \hat{C}_1 \cup \hat{C}_2$. 
Note that $\hat{C}_2$ is an irreducible type $(1,1)$ on $W_{\bar{\Bbbk}}$ passing through $\msp_2$ and satisfying $I_{\msp_0}(\hat{Q},\hat{C}_i) = 2$. Since $\hat{C}_1+\hat{C}_2$ is defined over \(\Bbbk\), so is $C_1+C_2 \coloneqq (\sigma \circ \tau)^{-1}_{\ast}(\hat{C}_1+\hat{C}_2)$. 
Then there exists a cylinder of $S$ contained in 
\begin{align*}
&S \setminus \Supp \Bigg(\pi_{\ast}\bigg(\sum_{i=1}^{m+2}E_i + E_1'' + E_2'' + L_1 + L_2 + C_1 + C_2\bigg)\Bigg) \\
&\qquad \qquad \simeq Y \setminus \Supp \bigg(Q+\sum_{i=1}^{m+2}E_i + E_1'' + E_2'' + L_1 + L_2 + C_1 + C_2\bigg) \\
&\qquad \qquad \simeq W \setminus \Supp (\hat{Q}+\hat{L}_1+\hat{L}_2+\hat{C}_1+\hat{C}_2),
\end{align*}
which implies that $S$ is cylindrical (see Example \ref{ex:(4)}). 
Therefore, this proves \((4)\).

\smallskip

Finally, we consider the case $\ell _S = m+4$. 
Then $Z$ is a smooth del Pezzo surface of degree $8$. 
In other words, $Z$ is a \(\Bbbk\)-form of $\F _1$ or $\P ^1 \times \P ^1$, where $\F_1$ is the Hirzebruch surface of degree $1$. Assume that $\ell_S = m+4$ and $m=2u-1$. We put $\bar{Q}\coloneqq \tau_{\ast}(Q)$, $\sum_{i=1}^{2n+3}\bar{E}_i\coloneqq \tau_{\ast}\left(\sum_{i=1}^{2u+3}E_i\right)$ and $\msp_i \coloneqq  \tau(E_i)$ for $i=1,\dots,2u+3$. 

Assume that $Z$ is a \(\Bbbk\)-form of $\F_1$. 
Let $\bar{E}_0$ be the minimal section of $Z_{\bar{\Bbbk}} \simeq \F_1$. Note that $\bar{E}_0$ is defined over \(\Bbbk\). Moreover, $\bar{E}_0 \simeq \P ^1_{\Bbbk}$. Hence, the general fiber $\bar{F}$ of the ruling $Z_{\bar{\Bbbk}} \simeq \F _1 \to \P ^1_{\bar{\Bbbk}}$ is defined over \(\Bbbk\). 
For $i=1,\dots , 2u+3$, let $\bar{E}_i''$ be an irreducible curve, which is linearly equivalent to $\bar{E}_0 + (2u+1)\bar{F}$ and passes through the points $\msp_1,\dots ,\msp_{i-1},\msp_{i+1},\dots ,\msp_{2u+3}$, on $Z_{\bar{\Bbbk}}$, and set $E_i'' \coloneqq \tau ^{-1}_{\bar{\Bbbk},\ast}(\bar{E_i}'')$. By construction, $E_i''$ is a $(-1)$-curve on $Y_{\bar{\Bbbk}}$ such that ${\rm Gal}(\bar{\Bbbk}/\Bbbk) \cdot \bigcup _{j=1}^{2u+3}E_j'' = \bigcup _{j=1}^{2u+3}E_j''$. Moreover, we have the following intersection numbers:
\begin{align*}
E_i'' \cdot Q = E_i'' \cdot E_j = 0, \text{ and } E_i'' \cdot E_j' = \delta _{i,j} \text{ for all } i,j. 
\end{align*}
Hence, $\sum_{i=1}^{2u+3}E_i''$ is defined over \(\Bbbk\) and can be further smoothly contracted over \(\Bbbk\). Let $\sigma \colon Y \to W$ be the contraction of $\sum_{i=1}^{2u+3}E_i''$, defined over \(\Bbbk\). Then $W$ is a \(\Bbbk\)-form of the Hirzebruch surface $\F _{2u-1}$ of degree $2u-1$. Moreover, $\hat{Q}\coloneqq \sigma _{\ast}(Q)$ is the minimal section and $\sigma _{\bar{\Bbbk},\ast}(E_i')$ is a closed fiber of the ruling $W_{\bar{\Bbbk}} \simeq \F _{2u-1} \to \P ^1_{\bar{\Bbbk}}$ for $i=1,\dots ,2u+3$. 
Then we see that there exists a cylinder of $S$ contained in 
\begin{align*}
S \setminus \Supp \Bigg(\pi _{\ast}\bigg(\sum_{i=1}^{2u+3}(E_i'+E_i'')\bigg)\Bigg) &\simeq Y \setminus \Supp \bigg(Q+\sum_{i=1}^{2u+3}(E_i'+E_i'')\bigg) \\&\simeq W \setminus \Supp \bigg(\hat{Q} + \sum_{i=1}^{2u+3}\hat{E}_i'\bigg),
\end{align*}
where $\sum_{i=1}^{2u+3}\hat{E}_i' \coloneqq \sigma_{\ast}\left(\sum_{i=1}^{2u+3}E_i'\right)$ (see Example \ref{ex:(0)}). 
Therefore, $S$ is cylindrical, and $S$ is rational by Lemma \ref{lem:Fm}. 

In what follows, we may assume that $Z$ is a \(\Bbbk\)-form of $\P ^1 \times \P ^1$. We may assume that $\bar{Q}_{\bar{\Bbbk}}$ is of type $(1,2)$ and each $\bar{E}_i'$ is of type $(1,0)$. For $i=1,\dots,2u+3$, let $\bar{E}_i''$ be an irreducible curve on $Z_{\bar{\Bbbk}} \simeq \P^1 \times\P^1$ of type $(0,1)$ passing through $\msp_i$. Note that the union $\sum_{i=1}^{2u+3}\bar{E}_i''$ is defined over $\Bbbk$. Moreover, we know that $\sum_{i=1}^{2u+3}E_i''\coloneqq \tau_{\ast}^{-1}\left(\sum_{i=1}^{2u+3}\bar{E}_i''\right)$ is a disjoint union and is defined over $\Bbbk$. Hence, we obtain the contraction $\sigma\colon Y \to W$ of $\sum_{i=1}^{2u+3}E_i''$, defined over $\Bbbk$. 
Then $W$ is a \(\Bbbk\)-form of the Hirzebruch surface $\F _{2u-1}$ of degree $2u-1$. Moreover, $\hat{Q}\coloneqq \sigma _{\ast}(Q)$ is the minimal section and $\sigma _{\bar{\Bbbk},\ast}(E_i')$ is a closed fiber of the ruling $W_{\bar{\Bbbk}} \simeq \F _{2u-1} \to \P ^1_{\bar{\Bbbk}}$ for $i=1,\dots ,2u+3$. 
Then we see that there exists a cylinder of $S$ contained in 
\begin{align*}
S \setminus \Supp \Bigg(\pi _{\ast}\bigg(\sum_{i=1}^{2u+3}(E_i'+E_i'')\bigg)\Bigg) \simeq W \setminus \Supp \bigg(\hat{Q} + \sum_{i=1}^{2u+3}\hat{E}_i'\bigg),
\end{align*}
where $\sum_{i=1}^{2u+3}\hat{E}_i' \coloneqq \sigma_{\ast}\left(\sum_{i=1}^{2u+3}E_i'\right)$ (see Example \ref{ex:(0)}). 
Therefore, $S$ is cylindrical, and rational by Lemma \ref{lem:odd points} combined with $Q(\Bbbk) \neq\emptyset$. 
\end{proof}

\subsection{Case: $n=m+5$}
In this subsection, we deal with the case when \(S\) is a \(\Bbbk\)-form of \(S_{m}^{m+5}\).

\begin{theorem}\label{thm:m+5}
Let $S$ be a $\Bbbk$-form of $S_m^{m+5}$, $\pi\colon Y \to S$ the minimal resolution, and $Q$ the exceptional curve. 
Then
\begin{enumerate}[(1)]
\item If $\ell _S \le m+1$, then $S$ is neither cylindrical nor rational. 
\item If $\ell _S > m+1$, then $\ell _S$ is one of $m+2$, $m+3$, $m+5$, $m+6$. 
\item If $\ell _S = m+2$ or $\ell_S = m+6$, then $S$ is both cylindrical and rational. 
\item If $\ell _S = m+3$ and $Q(\Bbbk) \neq\emptyset$, then $S$ is both cylindrical and rational. 
\item If $\ell _S = m+5$, then $S$ is rational. Moreover, if $Q(\Bbbk) \neq \emptyset$, then $S$ is cylindrical. 
\end{enumerate}
\end{theorem}

Before we present the proof of Theorem \ref{thm:m+5}, we note \(Q(\Bbbk)\neq \emptyset\) when \(m=3\). Hence, we have the following corollary.

\begin{corollary}\label{cor:m+5}
Let $S$ be a $\Bbbk$-form of $S_m^{m+5}$, $\pi \colon Y \to S$ the minimal resolution, and $Q$ the exceptional curve. 
Assume that $Q$ has a $\Bbbk$-rational point. 
Then
    \begin{align*}
    \text{$S$ is cylindrical} \iff \text{$S$ is rational} \iff \ell_{S} \ge 5. 
    \end{align*}
\end{corollary}

To prove Theorem \ref{thm:m+5}, we first establish the following lemmas.

\begin{lemma}\label{lem(1)}
The following assertions hold: 
\begin{enumerate}
\item There exists a $(-1)$-curve $E_0$ on $Y_{\bar{\Bbbk}}$ such that $E_0 \cdot Q = 2$. 
\item $-K_{Y_{\bar{\Bbbk}}} \sim Q+E_0$. 
\item For every $(-1)$-curve $E\neq E_0$ on $Y_{\bar{\Bbbk}}$, the pair $(E \cdot E_0,E \cdot Q)$ equals $(1,0)$ or $(0,1)$. 
\end{enumerate}
In particular, $E_0$ is defined over $\Bbbk$. Hence, $\ell_S \ge 1$. 
\end{lemma}

\begin{proof}
For (1) and (2), recall that, by the definition of $S_m^{m+5}$, $Y_{\bar{\Bbbk}}$ can be obtained by blowing up the Hirzebruch surface \(\F_m\) of degree $m$ at $(m+5)$ general points $\msp_1,\dots,\msp_{m+5}$. Let $M$ and $F$ be the minimal section and fiber component of $\F_m$, respectively. Then there exists an irreducible curve $B_0$ on $\F_m$ such that $B_0 \sim M + (m+2)F$ and $\msp_1,\dots,\msp_{m+5} \in B_0$. Let $E_0$ be the strict transform of $B_0$ on $Y_{\bar{\Bbbk}}$. Then $E_0$ is a $(-1)$-curve and satisfies $E_0 \cdot Q = 2$ since $Q$ is the strict transform of $M$. Moreover, by the construction of $E_0$, we know $-K_{Y_{\bar{\Bbbk}}} \sim Q+E_0$. 

For (3), let $E$ be a $(-1)$-curve on $Y_{\bar{\Bbbk}}$ such that $E \not= E_0$. Then we have $1 = E \cdot (-K_{Y_{\bar{\Bbbk}}})=E\cdot Q + E\cdot E_0$. Here, we note $E \cdot Q \ge 0$ and $E \cdot E_0 \ge 0$. Hence, we obtain the last assertion. 
\end{proof}

By the definition of $\ell_S$, there exist $(-1)$-curves $E_i$ for $i\in\{1,\ldots, \ell_S-1\}$ such that the union $E \coloneqq  E_1+\dots +E_{\ell_S-1}$ can be contracted over $\Bbbk$. 
Let $\sigma\colon Y \to Z$ be the contraction of $E_0+E$, defined  over $\Bbbk$, and $\tau\colon Y \to W$ the contraction of $E$, defined over $\Bbbk$. Set $\bar{E}_0 \coloneqq  \tau_{\ast}(E_0)$. 
By Lemma \ref{lem(1)}, there is no contractible curve over $\Bbbk$ on $Z$. Moreover, we have $(\sigma _{\ast}(Q))^2>0$. Hence, $Z$ is a $\Bbbk$-minimal smooth del Pezzo surface. 
On the other hand, note that $(-K_Z)^2 = (-K_Y)^2 + \ell_S = \ell_S - m + 3$. 

\begin{align*}
\xymatrix{
&&&W \ar[drr]^{{\rm cont}_{\bar{E}_0}}&& \\
& Y \ar[dl]_{\pi} \ar[urr]^{\tau} \ar[rrrr]^{\sigma} &&&& Z \\
S &&&&&
}
\end{align*}

\begin{lemma}\label{lem(2)}
If $\ell_S \le m+1$, then $S$ is neither cylindrical nor rational. 
\end{lemma}

\begin{proof}
By assumption, we have $(-K_Z)^2 \le 4$. Hence, $Z$ is a $\Bbbk$-minimal smooth del Pezzo surface of degree $\le 4$. 
By Theorem \ref{thm:Saw cyl iff}, $Z$ is neither cylindrical nor rational. In particular, $S$ is irrational. 
Moreover, since $Y$ is not cylindrical by Theorem \ref{thm:Saw cyl bir}, it follows that $S$ is not cylindrical.
\end{proof}

In what follows, we deal with the case $\ell_S \ge m+2$. Note that  $\bar{Q} \coloneqq  \tau_{\ast}(Q)$ has the self intersection number $(-K_Z)^2-4$. 

\begin{lemma}\label{lem(3)}
If $\ell_S = m+2$, then $S$ is cylindrical and rational. 
\end{lemma}

\begin{proof}
By assumption, $(-K_Z)^2 = 5$ and $W$ is a smooth del Pezzo surface of degree $4$. Then there exist exactly five $(-1)$-curves $\bar{E}_1',\dots ,\bar{E}_5'$ meeting $\bar{E}_0$. 
Furthermore, the union $\bar{E}_1'+\dots+\bar{E}_5'$ is a disjoint union. 
This yields a contraction $\psi\colon W \to V$ of $\bar{E}_1'+\dots +\bar{E}_5'$, defined over $\Bbbk$, where $V$ is a $\Bbbk$-form of $\P^2$. 
Note that $\hat{Q} \coloneqq \psi_{\ast}(\bar{Q})$ is a $1$-curve. By Lemma \ref{lem:line conic}, $V \simeq \P^2_{\Bbbk}$, so that $S$ is rational. 
Since $\hat{Q} \not= \emptyset$, we can take a $\Bbbk$-rational point $\hat{\msp}$ on $\hat{Q}$. 
For $i=1,\dots ,5$, set $\hat{\msp}_i \coloneqq  \psi(\bar{E}_i')$ on $V_{\bar{\Bbbk}}$ and let $\hat{L}_i$ be a line passing through two points $\hat{\msp}$ and $\hat{\msp}_i$.  
Note that the union $\hat{Q} + \sum _{j=1}^5\hat{L}_j$ is defined over $\Bbbk$ and consists only of lines on $V \simeq \P^2_{\Bbbk}$ passing through $\hat{\msp}$. 
Since $V \setminus \Supp\bigg(\hat{Q} + \sum _{j=1}^5\hat{L}_j \bigg)$ contains a cylinder of $V$ (see also Example \ref{ex:(1)}), we know that there exists a cylinder of $S$ contained in 
\begin{align*}
S \setminus \Supp\Bigg( \pi_{\ast}\bigg( \sum_{i=1}^{m+1}E_i + \sum_{j=1}^5(E_j'+L_j)\bigg)\Bigg) &\simeq Y \setminus \Supp \bigg(Q + \sum_{i=1}^{m+1}E_i +  \sum_{j=1}^5(E_j'+L_j)\bigg) \\
&\simeq V \setminus \Supp\bigg(\hat{Q} + \sum _{j=1}^5 \hat{L}_j \bigg),
\end{align*}
where $E_i' \coloneqq  \tau_{\ast}^{-1}(\bar{E}_i')$ and $L_i \coloneqq  (\psi \circ \tau)_{\ast}^{-1}(\hat{L}_i)$ for $i=1,\dots ,5$. 
Therefore, $S$ is rational and cylindrical. 
\end{proof}

\begin{lemma}\label{lem(4)}
If $\ell_S = m+3$ and $Q(\Bbbk) \not= \emptyset$, then $S$ is cylindrical and rational. 
\end{lemma}

\begin{proof}
By assumption, $(-K_Z)^2 = 6$ and $W$ is a smooth del Pezzo surface of degree $5$. Then there exist exactly three $(-1)$-curves $\bar{E}_1'$, $\bar{E}_2'$ and $\bar{E}_3'$ meeting $\bar{E}_0$. 
Furthermore, the union $\bar{E}_1'+\bar{E}_2'+\bar{E}_3'$ is a disjoint union. 
This yields a contraction $\psi\colon W \to V$ of $\bar{E}_1'+\bar{E}_2'+\bar{E}_3'$, defined over $\Bbbk$, where $V$ is a $\Bbbk$-form of $\P^1 \times \P^1$ with a $\Bbbk$-rational point $\hat{\msp}$ on $\hat{Q} \coloneqq  \psi_{\ast}(\bar{Q})$. 
By Lemma \ref{lem:p1p1}, we obtain that $S$ is rational. 
Set $\hat{\msp}_i \coloneqq  \psi(\bar{E}_i')$ on $V_{\bar{\Bbbk}}$ for $i=1,2,3$, and let $\hat{L}_1$ and $\hat{L}_2$ be irreducible curves of types $(1,0)$ and $(0,1)$ on $V_{\bar{\Bbbk}}$ passing through $\hat{\msp}$, respectively. 
Note that the union $\hat{L}_1 + \hat{L}_2$ is defined over $\Bbbk$, so that $\hat{\msp}_i \in \hat{L}_1 \cup \hat{L}_2$ for some $i=1,2,3$ if and only if $\hat{\msp}_i \in \hat{L}_1 \cup \hat{L}_2$ for every $i=1,2,3$. 
Set $E_i' \coloneqq  \tau_{\ast}^{-1}(\bar{E}_i')$ for $i=1,2,3$, and $L_i \coloneqq  (\psi \circ\tau)_{\ast}^{-1}(\hat{L}_i)$ for $i=1,2$. 
We consider the following two cases separately. 

\smallskip

{\bf Case 1}: $\hat{\msp}_1,\hat{\msp}_2,\hat{\msp}_3 \in \hat{L}_1 \cup \hat{L}_2$. 
Since $\hat{Q}$, $\hat{L}_1$ and $\hat{L}_2$ intersect at only one point $\hat{\msp}$, $V \setminus \Supp(\hat{Q} + \hat{L}_1+\hat{L}_2)$ contains a cylinder of $V$ (see Example \ref{ex:(4)}). 
Hence, there exists a cylinder of $S$ contained in 
\begin{align*}
&S \setminus \Supp\Bigg( \pi_{\ast}\bigg( \sum_{i=1}^{m+2}E_i + \sum_{j=1}^3E_j' + L_1+L_2\bigg)\Bigg)\\
&\quad \simeq Y \setminus \Supp \bigg(Q + \sum_{i=1}^{m+2}E_i +  \sum_{j=1}^3E_j' + L_1+L_2\bigg) \\
&\quad \simeq V \setminus \Supp(\hat{Q} + \hat{L}_1+\hat{L}_2). 
\end{align*}
Therefore, $S$ is rational and cylindrical. 

\smallskip

{\bf Case 2}: $\hat{\msp}_1,\hat{\msp}_2,\hat{\msp}_3 \not\in \hat{L}_1 \cup \hat{L}_2$. 
For $i=1,2,3$, let $\hat{C}_i$ be an irreducible curve of type $(1,1)$ on $V_{\bar{\Bbbk}}$ passing through two points $\hat{\msp}$ and $\hat{\msp}_i$ such that $I_{\hat{\msp}}(\hat{C}_i,\hat{Q}) = 2$.
Note that the union $\sum _{j=1}^3\hat{C}_j$ is defined over $\Bbbk$, and $\hat{C}_i = \hat{C}_j$ if and only if $\hat{\msp}_j \in \hat{C}_i$ for $i,j=1,2,3$. 
Since $\hat{Q}$, $\hat{C}_1$, $\hat{C}_2$, $\hat{C}_3$, $\hat{L}_1$ and $\hat{L}_2$ intersect at exactly one point $\hat{\msp}$, $V \setminus \Supp\bigg(\hat{Q} + \sum_{j=1}^3\hat{C}_j + \hat{L}_1+\hat{L}_2 \bigg)$ contains a cylinder of $V$  (see Example \ref{ex:(4)}). 
Hence, there exists a cylinder of $S$ contained in 
\begin{align*}
    & S \setminus \Supp \Bigg( \pi_{\ast}\bigg( \sum_{i=1}^{m+2}E_i + \sum_{j=1}^3(E_j'+C_j) +L_1+L_2\bigg)\Bigg)\\
    &\quad \simeq Y \setminus \Supp \bigg(Q + \sum_{i=1}^{m+2}E_i +  \sum_{j=1}^3(E_j'+C_j) + L_1+L_2\bigg) \\
    &\quad \simeq V \setminus \Supp \bigg(\hat{Q} + \sum_{j=1}^3\hat{C}_j + \hat{L}_1+\hat{L}_2 \bigg),
\end{align*}
where $C_j \coloneqq  (\psi \circ \tau)_{\ast}^{-1}(\hat{C}_j)$ for $j=1,2,3$. 
Therefore, $S$ is rational and cylindrical. 
\end{proof}

\begin{lemma}\label{lem(5)}
The case $\ell_S = m+4$ does not occur. 
\end{lemma}

\begin{proof}
Suppose that $\ell_S = m+4$. Then $Z$ is a smooth del Pezzo surface of degree $7$. However, this is impossible since $Z$ is $\Bbbk$-minimal. 
\end{proof}

\begin{lemma}\label{lem(6)}
If $\ell_S = m+5$, then $S$ is rational. Moreover, if $Q(\Bbbk) \not= \emptyset$, then $S$ is cylindrical. 
\end{lemma}

\begin{proof}
By assumption, $(-K_Z)^2 = 8$ and $W$ is a smooth del Pezzo surface of degree $7$. We notice that $Z$ is a $\Bbbk$-form of $\P^1 \times \P^1$ since $Z$ is $\Bbbk$-minimal. Then there exist exactly two $(-1)$-curves $\bar{E}_1'$ and $\bar{E}_2'$ meeting $\bar{E}_0$. 
Note that the union $\bar{E}_1'+\bar{E}_2'$ is a disjoint union. 
This yields a contraction $\psi\colon W \to V$ of $\bar{E}_1'+\bar{E}_2'$, defined over $\Bbbk$, where $V$ is a $\Bbbk$-form of $\P^2$. 
Note that the curve $\hat{Q} \coloneqq \psi_{\ast}(\bar{Q})$ has self-intersection number \(4\), hence it is a conic on \(V\). By Lemma \ref{lem:line conic}, $V \simeq \P^2_{\Bbbk}$, so that $S$ is rational. 
In what follows, we further assume that $Q(\Bbbk) \not= \emptyset$. Then $\hat{Q}$ has a $\Bbbk$-rational point $\hat{\msp}$. 
For $i=1,2$, set $\hat{\msp}_i \coloneqq  \psi(\bar{E}_i')$ on $V_{\bar{\Bbbk}}$ and let $\hat{Q}_i$ be an irreducible conic passing through two points $\hat{\msp}$ and $\hat{\msp}_i$ such that $I_{\hat{\msp}}(\hat{Q},\hat{Q}_i) = 4$.
Note that $\hat{Q}_1+\hat{Q}_2$ is defined over $\Bbbk$, and $\hat{Q}_1 = \hat{Q}_2$ if and only if $\hat{\msp}_2 \in \hat{Q}_1$. 
Then there exists a unique line $\hat{T}$ on $V$, which is the tangent line of $\hat{Q}$, $\hat{Q}_1$ and $\hat{Q}_2$ at $\hat{\msp}$. 
Since $V \setminus \Supp\bigg(\hat{Q} + \sum_{j=1}^2\hat{Q}_j + \hat{T}\bigg)$ contains a cylinder of $V$ (see Example \ref{ex:(2)}), there exists a cylinder of $S$ contained in 
\begin{align*}
&S \setminus \Supp\Bigg( \pi_{\ast}\bigg( \sum_{i=1}^{m+4}E_i + \sum_{j=1}^2(E_j'+Q_j) + T\bigg)\Bigg)\\ 
 &\quad \simeq Y \setminus \Supp \bigg(Q + \sum_{i=1}^{m+4}E_i + \sum_{j=1}^2(E_j'+Q_j) + T\bigg) \\
 &\quad \simeq V \setminus \Supp\bigg(\hat{Q} + \sum_{j=1}^2\hat{Q}_j + \hat{T}\bigg), 
\end{align*}
where $E_i' \coloneqq  \tau_{\ast}^{-1}(\bar{E}_i')$ and $Q_i \coloneqq  (\psi \circ \tau)_{\ast}^{-1}(\hat{Q}_i)$ for $i=1,2$, and $T \coloneqq  (\psi \circ \tau)_{\ast}^{-1}(\hat{T})$. 
Therefore, $S$ is rational and cylindrical. 
\end{proof}

\begin{lemma}\label{lem(7)}
If $\ell_S = m+6$, then $S$ is rational and cylindrical. 
\end{lemma}

\begin{proof}
By assumption, $Z$ is a $\Bbbk$-form of $\P^2$ with a $1$-curve $\sigma_{\ast}(E_0)$. 
By Lemma \ref{lem:line conic}, $Z \simeq \P^2_{\Bbbk}$, so that $S$ is rational. 
Notice that $W$ is the Hirzebruch surface $\F_1$ of degree $1$. 
Then $\bar{E}_0$ and $\bar{Q}$ intersect at exactly two points $\bar{\msp}_1$ and $\bar{\msp}_2$. 
Let $\bar{F}_1$ and $\bar{F}_2$ be distinct fibers of the ruling on $W = \F_1$ passing through $\bar{\msp}_1$ and $\bar{\msp}_2$, respectively. 
Note that the disjoint union $\bar{F}_1+\bar{F}_2$ is defined over $\Bbbk$ since the union $\bar{\msp}_1+\bar{\msp}_2$ is defined over $\Bbbk$. 
Since $W \setminus \Supp(\bar{Q} + \bar{F}_1 + \bar{F}_2)$ contains a cylinder of $W$ (see Example \ref{ex:(3)}), there exists a cylinder of $S$ contained in 
\begin{align*}
S \setminus \Supp\left( \pi_{\ast}\left( \sum_{i=1}^{m+5}E_i +  F_1 + F_2\right)\right) &\simeq Y \setminus \Supp \left(Q + \sum_{i=1}^{m+5}E_i +  F_1 + F_2\right) \\
&\simeq W \setminus \Supp(\bar{Q} + \bar{F}_1 + \bar{F}_2),
\end{align*}
where $F_i \coloneqq  \tau_{\ast}^{-1}(\bar{F}_i)$ for $i=1,2$. 
Therefore, $S$ is rational and cylindrical. 
\end{proof}

\begin{proof}[Proof of Theorem \ref{thm:m+5}]
    This follows from Lemmas \ref{lem(2)}, \ref{lem(3)}, \ref{lem(4)}, \ref{lem(5)}, \ref{lem(6)} and \ref{lem(7)}.
\end{proof}

Restricting Theorem \ref{thm:combined} to $\Bbbk$-forms $S$ of Picard rank one, we obtain the following corollary.

\begin{corollary}\label{cor:rank1}
    Let \(S\) be a \(\Bbbk\)-form of \(S_{m}^{n}\) for some \(m\geq 2\), and \(Q\) the unique \((-m)\)-curve on the minimal resolution defined over \(\Bbbk\). Assume that \(S\) is of Picard rank one and \(Q(\Bbbk)\neq \emptyset\). Then \(S\) is cylindrical if and only if \(n\leq m+3\). 
\end{corollary}

\begin{proof}
If \(n\leq m+3\), then $S$ is cylindrical by Theorem \ref{thm:intermediate}. Hence, we treat the case \(n\geq m+4\) in what follows. We note that $\rho_{\Bbbk}(Y)=2$. 

Assume that \(n= m+4\). Let $\varphi\colon Y \to B$, $\tau\colon Y \to Z$ and $\psi\colon Z \to B$ be as in the proof of Theorem \ref{thm:m+4}. Here, note that $B$ is a projective curve. Suppose that $\ell_S > 0$. Then $\rho_{\Bbbk}(Z) < \rho_{\Bbbk}(Y) = 2$ since $\tau$ is not the identity morphism. On the other hand, since $Z$ admits a fibration structure, we have $\rho_{\Bbbk}(Z) \ge 2$, which is a contradiction. Thus, $\ell_S = 0$, which implies that $S$ is not cylindrical by Theorem \ref{thm:m+4}. 

Assume that \(n> m+4\). Then $n=m+5$ and $m\in \{2,3\}$. 
By Lemma \ref{lem(1)}, there exists a $(-1)$-curve $E_0$ defined over $\Bbbk$ such that $E_0 \cdot Q = 2$. Let $\psi\colon Y \to Z$ be the contraction of $E_0$ defined over $\Bbbk$. Then $Z$ is a smooth del Pezzo surface of Picard rank one since $\rho_{\Bbbk}(Y)=2$. Moreover, we have $(-K_Z)^2 = (-K_Y)^2 + 1 \le 2$. By Theorem \ref{thm:Saw cyl iff}, $Z$ is not cylindrical. Hence, $Y$ is not cylindrical by Theorem \ref{thm:Saw cyl bir}. Therefore, $S$ is not cylindrical. This completes the proof. 
\end{proof}

\begin{remark}
Let \(S\) be a \(\Bbbk\)-form of the del Pezzo surface \(S_{1}^{n}\). Assume that \(S\) is of Picard rank one and \(S(\Bbbk)\neq \emptyset\). In this case, \(S\) is cylindrical if and only if \(n\leq 4\) by {\cite{DK18}} (see also Theorem \ref{thm:Saw cyl bir}). Corollary \ref{cor:rank1} can be considered a generalization of this result.
\end{remark}

\section{Embedding models}\label{sect:main2}

Over an algebraically closed field $\bar{\Bbbk}$, the surface $S_m^{m+4}$ is isomorphic to a hypersurface or a complete intersection in a weighted projective space (see \cite[Lemma 4.4]{KKW25}). We first show that any $\Bbbk$-form of the surface \(S_{m}^{m+4}\) is also embedded in a weighted projective space as either a hypersurface or a complete intersection.

\smallskip

First, we recall the orbifold Riemann--Roch formula, which is due to \cite{YPG}. 

\begin{theorem}[{\cite[Chapter III]{YPG}}]\label{thm:YPG}
    Let \(S\) be a normal projective surface and \(D\) a \(\Q\)-Cartier divisor on \(S\). Then we have the following formula:
    \begin{align*}
        \chi(S,\OO_S(D))=\chi(\OO_S)+\frac{1}{2}D\cdot(D-K_S)+\sum_{\msp\in \Sing(S)}c_{\msp}^{D}(j),
    \end{align*}
    where \(c_{\msp}^{D}(j)\) is the correction term at each singular point \(\msp\) that depends on the type of quotient singularity at \(\msp\) and the local behavior of \(D\), and \(j\) is the local index of \(D\) at \(\msp\). 
\end{theorem}

The correction term \(c_{\msp}^{D}\) is defined recursively; see \cite[Section 8]{YPG}. If \(S\) is a normal projective surface with exactly one quotient singularity of type \(\frac{1}{m}(1,1)\) such that \(-K_S\) is ample \(\Q\)-Cartier, then we obtain the following.
\begin{lemma}\label{lem:cp term}
    Let \(S\) be a normal projective surface with exactly one quotient singularity of type \(\frac{1}{m}(1,1)\) such that \(-K_S\) is ample \(\Q\)-Cartier. Then we have
    \begin{align*}
        c_{\msp}^{-K_S}(j)\coloneqq
        \begin{dcases}
            \sigma_{t}\left(\frac{1}{m}(1,1)\right)-\sigma_{0}\left(\frac{1}{m}(1,1)\right) &\text{ if } t\not\equiv 0 \text{ (mod } m),\\
            0 & \text{ if } t\equiv 0 \text{ (mod } m),
        \end{dcases}
    \end{align*}
    where \(t\) is an integer \(0\leq t\leq m-1\) that is defined as a solution of \(-2j\equiv t \text{ (mod } m)\). Moreover, a direct calculation gives
    \begin{align*}
        (\sigma_{t}-\sigma_0)\left(\frac{1}{m}(1,1)\right)=\frac{1}{m}\left(-\frac{m-1}{2}+\frac{(m-t+1)(t-1)}{2}\right) \text{ for all } 1\leq t\leq m-1.
    \end{align*}
\end{lemma}
We will use this lemma in the proof of Theorem \ref{thm:embedding}.

\smallskip

\begin{theorem}\label{thm:embedding}
    Let \(S\) be a \(\Bbbk\)-form of \(S_{m}^{m+4}\). If \(m=2u-1\) with \(u\geq 2\), then \(S\) is isomorphic to a weighted complete intersection of two hypersurfaces of degree \(2u\) embedded in \(\P(1,1,u,u,2u-1)\). If \(m=2u\) with \(u\geq 1\), then \(S\) is isomorphic to a weighted hypersurface of degree \(2u+2\) in \(\P(1,1,u,u+1)\).
\end{theorem}
\begin{proof}[Proof of Theorem \ref{thm:embedding}]
    By the Kawamata--Viehweg vanishing theorem, we have \(H^{q}(S,-jK_S)=0\) for \(q=1,2\) and \(j>0\). Hence, we have the following orbifold Riemann--Roch formula by Theorem \ref{thm:YPG}:
    \begin{align*}
        h^{0}(S,-jK_S)=\chi(S,-jK_S)=1+\frac{j(j+1)}{2}(-K_S)^2+c_{\msp}^{-K_S}(j).
    \end{align*}
    By Lemma \ref{lem:cp term}, we obtain 
    \begin{align*}
        h^{0}(S,-K_S)=1+\frac{4}{m}+c_{\msp}^{-K_S}(1)=1+\frac{4}{m}+\frac{m-4}{m}=2.
    \end{align*}
    If \(m=2u-1\) for \(u\geq 2\), then we have
    \begin{align*}
        h^{0}(S,-uK_S)&=1+\frac{2u(u+1)}{2u-1}+c_{\msp}^{-K_S}(u)=1+\frac{2u(u+1)}{2u-1}+\frac{u-2}{2u-1}=u+3,\\
        h^{0}(S,-(2u-1)K_S)&=1+4u+c_{\msp}^{-K_S}(2u-1)=4u+1, ~ \text{ and }\\
        h^{0}(S,-2uK_S)&=1+\frac{4u(2u+1)}{2u-1}+c_{\msp}^{-K_S}(2u)=1+\frac{4u(2u+1)}{2u-1}+\frac{2u-5}{2u-1}=4u+6.
    \end{align*}
    This implies that the anticanonical model of \(S\) is \(\Proj \left(\Bbbk[x_0,x_1,x_2,x_3,x_4]/(f_1,f_2)\right)\), where \(\deg(x_0)=\deg(x_1)=1, \deg(x_2)=\deg(x_3)=u, \deg(x_4)=2u-1\) and \(f_1\), \(f_2\) are degree \(2u\) quasi-homogeneous polynomials. Since \(-K_S\) is ample, \(S\) is isomorphic to its anticanonical model. Hence, \(S\) is isomorphic to a weighted complete intersection of two degree \(2u\) hypersurfaces embedded in \(\P(1,1,u,u,2u-1)\).

    If \(m=2u\) for \(u\geq 1\), then we have
    \begin{align*}
        h^{0}(S,-uK_S)&=1+(u+1)+c_{\msp}^{-K_S}(u)=u+2,\\
        h^{0}(S,-(u+1)K_S)&=1+\frac{(u+1)(u+2)}{u}+c_{\msp}^{-K_S}(u+1)\\
        &=1+\frac{(u+1)(u+2)}{u}+\frac{2u-4}{2u}\\&=u+5,~ \text{ and }\\
        h^{0}(S,-(2u+2)K_S)&=1+\frac{2(u+1)(2u+3)}{u}+c_{\msp}^{-K_S}(2u+2)\\
        &=1+\frac{2(u+1)(2u+3)}{u}+\frac{2u-6}{2u}\\&=4u+13.
    \end{align*}
    Thus, we obtain that the anticanonical model of \(S\) is \(\Proj \left(\Bbbk[x_0,x_1,x_2,x_3]/(f)\right)\), where \(\deg(x_0)=\deg(x_1)=1, \deg(x_2)=u, \deg(x_3)=u+1\) and \(f\) is a degree \(2u+2\) quasi-homogeneous polynomial. Since \(-K_S\) is ample, \(S\) is isomorphic to its anticanonical model. Hence, \(S\) is isomorphic to a weighted hypersurface of degree \(2u+2\) embedded in \(\P(1,1,u,u+1)\).
\end{proof}

\section{Applications}\label{sect:examples}


Let $u$ be an integer with $u \ge 2$. 
In this section, we consider examples of \(\Bbbk\)-forms of \(S_{2u-1}^{2u+3}\). Moreover, we shall discuss whether such $\Bbbk$-forms contain a cylinder or not. 

By Theorem \ref{thm:embedding}, every \(\Bbbk\)-form of \(S_{2u-1}^{2u+3}\) is isomorphic to a weighted complete intersection of two degree \(2u\) hypersurfaces in \(\P(1,1,u,u,2u-1)\). 
After a suitable weighted coordinate change, we may write such a surface \(S\) in the form
\begin{equation}\label{eq:odd-ci-model}
\begin{dcases}
xw+zt=0,\\
yw+z^2+t^2+f(x,y)z+g(x,y)t+h(x,y)=0,
\end{dcases}
\end{equation}
where \(f(x,y)\), \(g(x,y)\), and \(h(x,y)\) are homogeneous polynomials of degrees \(u\), \(u\), and \(2u\), respectively. For simplicity, we restrict to the case \(f(x,y)=g(x,y)=0\).

\smallskip

Since the variables \(x\) and \(y\) have weight \(1\), the anticanonical system is generated by their restrictions, and its members are the hyperplane sections cut out by \(y=ax\) together with the member \(x=0\). Let \(H_a\) denote the divisor on \(S\) given by \(y=ax\), and let \(H_x\) be the divisor cut out by \(x=0\). On the affine chart \(S\setminus H_x\), the section \(H_a\) is described by
\begin{equation*}
\begin{dcases}
w+zt=0,\\
aw+z^2+t^2+h(1,a)=0.
\end{dcases}
\end{equation*}
Eliminating \(w\), we obtain
\[
H_a\cap (S\setminus H_x)\simeq \left\{-azt+z^2+t^2+h(1,a)=0\right\}.
\]
Therefore, \(H_a\) is reducible if and only if the associated binary quadratic form is degenerate, namely if and only if
\begin{equation*}
\Bigl(1-\frac{a^2}{4}\Bigr)h(1,a)=0.
\end{equation*}
Whenever \(H_a\) is reducible, the strict transform of \(H_a\) on the minimal resolution contains \((-1)\)-curves meeting the distinguished curve \(Q\). In this way, the zeros of \(h(1,a)\) control the possible values of the invariant \(\ell_S\). For instance, if
\[
h(x,y)=x^{2u}+y^{2u},
\]
then \(h(1,a)=1+a^{2u}\). Thus, the resulting \(\Q\)-form satisfies \(\ell_S\le 2u-1\). By Corollary \ref{cor:m+4}, this implies that the surface is neither rational nor cylindrical over \(\Q\). More generally, by choosing the polynomial \(h(x,y)\) appropriately, one obtains examples with various prescribed values of \(\ell_S\).

\smallskip

We now pass from these weighted projective models to a relative construction that produces a non-isotrivial Fano threefold $g\colon X\rightarrow \P^1$ such that $X$ contains a vertical cylinder with respect to $g$. 

\begin{proof}[Proof of Theorem \ref{thm:Fano}]
Let \(B\coloneqq \P^1_{\C}\) with homogeneous coordinates \([s_0:s_1]\), and let \(L'\coloneqq \OO_B(1)\). Consider the graded \(\OO_B\)-algebra \(\mathcal{A}\) generated by
\[
x,\ y \in H^0(B,\OO_B),\qquad
z,\ t \in H^0(B,L'),\qquad
w \in H^0(B,(L')^{\otimes 2}),
\]
with weighted degrees \(1,1,u,u,2u-1\), respectively. Set
\[
p\colon \mathcal{Y}\coloneqq \Proj_B(\mathcal{A})\longrightarrow B.
\]
Every fiber of \(p\) is isomorphic to \(\P(1,1,u,u,2u-1)\). Denote by \(H\) the relative weighted hyperplane class on \(\mathcal{Y}\), and by \(L\) the pullback of \(L'\) from \(B\). Then
\[
H=[x=0]=[y=0],\qquad
uH+L=[z=0]=[t=0],\qquad
(2u-1)H+2L=[w=0].
\]

Fix distinct general complex numbers \(\alpha_1,\dots,\alpha_{2u-1}\) with \(\alpha_i\neq \pm 2\), and define
\[
Q(x,y)\coloneqq \prod_{i=1}^{2u-1}(y-\alpha_i x).
\]
Inside \(\mathcal{Y}\), consider the codimension two subvariety \(X\) defined by
\begin{equation*}
\begin{dcases}
F_1\coloneqq xw+zt=0,\\
F_2\coloneqq yw+z^2+t^2+s_0(s_0y-s_1x)Q(x,y)=0.
\end{dcases}
\end{equation*}
Both \(F_1\) and \(F_2\) are global sections of \(\OO_{\mathcal{Y}}(2uH+2L)\). Thus, \(X\) is a relative weighted complete intersection over \(B\). By adjunction,
\[
-K_X=(H+2L)|_X.
\]
Since \(H\) is relatively ample over \(B\) and \(L\) is the pullback of an ample divisor on \(B\), the divisor \(H+2L\) is ample on \(\mathcal{Y}\), and hence \(-K_X\) is ample on \(X\). Thus, \(X\) is a Fano threefold.

\smallskip

Let \(g\coloneqq p|_X\colon X\to B\). On the affine chart \(s_0\neq 0\), set \(\tau=s_1/s_0\). Then the fiber \(X_\tau\) is given by
\begin{equation*}
\begin{dcases}
xw+zt=0,\\
yw+z^2+t^2+Q(x,y)(y-\tau x)=0.
\end{dcases}
\end{equation*}
This is the same normal form as \eqref{eq:odd-ci-model}, with \(h(x,y)=Q(x,y)(y-\tau x)\).

Hence, \(X_\tau\) is a \(\C(B)\)-form of \(S_{2u-1}^{2u+3}\). Moreover, the factor \(h(1,a)=Q(1,a)(a-\tau)\) has \(2u\) roots. Thus, the generic fiber satisfies \(\ell_{X_\eta}\ge 2u\). By Corollary \ref{cor:m+4}, the generic fiber \(X_\eta\) is cylindrical. Therefore, \(X\) contains a vertical cylinder with respect to \(g\) by \cite[Lemma 3]{DK18}.

\smallskip

The family \(g\colon X\to B\) is also non-isotrivial. Indeed, by the description above, a member \(H_a\in |-K_{X_\tau}|\) is reducible if and only if
\[
\Bigl(1-\frac{a^2}{4}\Bigr)Q(1,a)(a-\tau)=0.
\]
Thus the set of reducible anticanonical members is
\[
R_\tau=\{\pm 2,\alpha_1,\dots,\alpha_{2u-1},\tau\}.
\]
Since \(\alpha_1,\dots,\alpha_{2u-1}\) are chosen to be distinct and general, the subgroup of \(\mathrm{PGL}_2\) preserving \(\{\pm 2,\alpha_1,\dots,\alpha_{2u-1}\}\) is trivial. Hence, any isomorphism \(X_\tau\simeq X_{\tau'}\) must induce the identity on \(|-K_{X_\tau}|\cong \P^1\), and therefore \(\tau=\tau'\). Consequently, \(g\colon X\to B\) is a non-isotrivial Fano threefold carrying a vertical cylinder. 

\smallskip

Theorem \ref{thm:Fano} is thus completed. 
\end{proof}

\section*{Acknowledgements}
The authors would like to thank Professors Adrien Dubouloz and Takashi Kishimoto for carefully reading this paper and for their valuable comments. The second author is partially supported by the Basic Science Research Program through the National Research Foundation of Korea (NRF) funded by the Ministry of Education (No. RS-2023-00237440 and 2021R1A6A1A10039823) and by Samsung Science and Technology Foundation under Project Number SSTF-BA2302-03. The third author is supported by JSPS KAKENHI Grant Number JP24K22823 and JP25K17222.  

\bibliographystyle{habbvr}
\bibliography{biblio}

@book {KSC,
    AUTHOR = {Koll\'ar, J\'anos and Smith, Karen E. and Corti, Alessio},
     TITLE = {Rational and nearly rational varieties},
    SERIES = {Cambridge Studies in Advanced Mathematics},
    VOLUME = {92},
 PUBLISHER = {Cambridge University Press, Cambridge},
      YEAR = {2004},
     PAGES = {vi+235},
      ISBN = {0-521-83207-1},
   MRCLASS = {14M20 (14E08)},
  MRNUMBER = {2062787},
MRREVIEWER = {Alexandr\ V.\ Pukhlikov},
       DOI = {10.1017/CBO9780511734991},
       URL = {https://doi.org/10.1017/CBO9780511734991},
}

@article {DK18,
    AUTHOR = {Dubouloz, Adrien and Kishimoto, Takashi},
     TITLE = {Cylinders in del {P}ezzo fibrations},
   JOURNAL = {Israel J. Math.},
  FJOURNAL = {Israel Journal of Mathematics},
    VOLUME = {225},
      YEAR = {2018},
    NUMBER = {2},
     PAGES = {797--815},
      ISSN = {0021-2172,1565-8511},
   MRCLASS = {14R25 (14J26 14R05 14R10)},
  MRNUMBER = {3805666},
MRREVIEWER = {Tatiana\ M.\ Bandman},
       DOI = {10.1007/s11856-018-1679-z},
       URL = {https://doi.org/10.1007/s11856-018-1679-z},
}

@article {Saw23a,
    AUTHOR = {Sawahara, Masatomo},
     TITLE = {Cylinders in weak del {P}ezzo fibrations},
   JOURNAL = {Transform. Groups},
  FJOURNAL = {Transformation Groups},
    VOLUME = {28},
      YEAR = {2023},
    NUMBER = {1},
     PAGES = {413--437},
      ISSN = {1083-4362,1531-586X},
   MRCLASS = {14R10 (14D06 14J26)},
  MRNUMBER = {4552181},
MRREVIEWER = {Adrien\ Dubouloz},
       DOI = {10.1007/s00031-022-09730-y},
       URL = {https://doi.org/10.1007/s00031-022-09730-y},
}

@article {Saw23b,
    AUTHOR = {Sawahara, Masatomo},
     TITLE = {Notes on cylinders in smooth projective surfaces},
   JOURNAL = {Geom. Dedicata},
  FJOURNAL = {Geometriae Dedicata},
    VOLUME = {217},
      YEAR = {2023},
    NUMBER = {1},
     PAGES = {Paper No. 6, 14},
      ISSN = {0046-5755,1572-9168},
   MRCLASS = {14R25 (14E05 14E30 14J26)},
  MRNUMBER = {4505008},
MRREVIEWER = {Jing\ Zhang},
       DOI = {10.1007/s10711-022-00741-3},
       URL = {https://doi.org/10.1007/s10711-022-00741-3},
}

@incollection {YPG,
    AUTHOR = {Reid, Miles},
     TITLE = {Young person's guide to canonical singularities},
 BOOKTITLE = {Algebraic geometry, {B}owdoin, 1985 ({B}runswick, {M}aine,
              1985)},
    SERIES = {Proc. Sympos. Pure Math.},
    VOLUME = {46, Part 1},
     PAGES = {345--414},
 PUBLISHER = {Amer. Math. Soc., Providence, RI},
      YEAR = {1987},
      ISBN = {0-8218-1476-1},
   MRCLASS = {14E30 (14B05 14E05 14J10)},
  MRNUMBER = {927963},
MRREVIEWER = {Eckart\ Viehweg},
       DOI = {10.1090/pspum/046.1/927963},
       URL = {https://doi.org/10.1090/pspum/046.1/927963},
}

@article {CP20,
    AUTHOR = {D. Cavey and T. Prince},
     TITLE = {Del {P}ezzo surfaces with a single {$1/k(1,1)$} singularity},
   JOURNAL = {J. Math. Soc. Japan},
  FJOURNAL = {Journal of the Mathematical Society of Japan},
    VOLUME = {72},
      YEAR = {2020},
    NUMBER = {2},
     PAGES = {465--505},
      ISSN = {0025-5645,1881-1167},
   MRCLASS = {14J26 (14M25)},
  MRNUMBER = {4090344},
MRREVIEWER = {Husney\ Parvez\ Sarwar},
       DOI = {10.2969/jmsj/79337933},
       URL = {https://doi.org/10.2969/jmsj/79337933},
}

@Misc{S,
  author       = {Shokurov, V. V.},
  howpublished = {\href{https://arxiv.org/abs/2012.06495}{arXiv:2012.06495}},
  title        = {Existence and boundedness of $n$-complements},
  year         = {2020},
  copyright    = {arXiv.org perpetual, non-exclusive license},
  publisher    = {arXiv},
}

@Article{B,
  author    = {C. Birkar},
  journal   = {Ann. of Math. (2)},
  title     = {{Anti-pluricanonical systems on Fano varieties}},
  year      = {2019},
  number    = {2},
  pages     = {345--463},
  volume    = {190},
  doi       = {10.4007/annals.2019.190.2.1},
  fjournal  = {Annals of Mathematics},
  publisher = {Department of Mathematics of Princeton University},
  url       = {https://doi.org/10.4007/annals.2019.190.2.1},
}

@article{G,
	author = {Y. Gongyo},
	Journal = {J. Algebraic Geom.},
	Fjournal = {Journal of Algebraic Geometry},
	pages = {549--564},
	title = {Abundance theorem for numerically trivial log canonical divisors of semi-log canonical pairs},
	volume = {22},
	year = {2012}
}

@Article{Z,
  author  = {D.-Q. Zhang},
  journal = {J. Math. Kyoto. Univ.},
  title   = {Logarithmic {E}nriques surfaces},
  year    = {1991},
  issn    = {2156-2261},
  volume  = {31},
  doi     = {10.1215/kjm/1250519795},
}

@Article{BCHM,
  author   = {Birkar, C. and Cascini, P. and Hacon, C. D. and McKernan, J.},
  journal  = {J. Am. Math. Soc.},
  title    = {Existence of minimal models for varieties of log general type},
  year     = {2010},
  issn     = {0894-0347},
  number   = {2},
  pages    = {405--468},
  volume   = {23},
  doi      = {10.1090/S0894-0347-09-00649-3},
  language = {English},
  zbl      = {1210.14019},
  zbmath   = {5775673},
}

@Book{L,
  author    = {L. B\v{a}descu},
  publisher = {Universitext, Springer-Verlag, New York},
  title     = {Algebraic surfaces},
  year      = {2001},
}

@Book{Poo17,
  author    = {Poonen, Bjorn},
  publisher = {Amer. Math. Soc., Providence, RI},
  title     = {Rational Points on Varieties},
  year      = {2017},
  isbn      = {978-1-4704-3773-2; 978-1-4704-4315-3},
  series    = {Grad. Stud. Math.},
  volume    = {186},
  fseries   = {Graduate Studies in Mathematics},
  doi       = {10.1090/gsm/186},
  language  = {English},
  zbl       = {1387.14004},
  zbmath    = {},
}

@article {CPW17,
    AUTHOR = {Cheltsov, Ivan and Park, Jihun and Won, Joonyeong},
     TITLE = {Cylinders in del {P}ezzo surfaces},
   JOURNAL = {Int. Math. Res. Not. IMRN},
  FJOURNAL = {International Mathematics Research Notices. IMRN},
      YEAR = {2017},
    NUMBER = {4},
     PAGES = {1179--1230},
      ISSN = {1073-7928,1687-0247},
   MRCLASS = {14J26 (14C20 14R25)},
  MRNUMBER = {3658164},
MRREVIEWER = {Tatiana\ M.\ Bandman},
       DOI = {10.1093/imrn/rnw063},
       URL = {https://doi.org/10.1093/imrn/rnw063},
}

@article {KM78,
    AUTHOR = {T. Kambayashi and M. Miyanishi},
     TITLE = {On flat fibrations by the affine line},
   JOURNAL = {Illinois J. Math.},
  FJOURNAL = {Illinois Journal of Mathematics},
    VOLUME = {22},
    NUMBER = {4},
    YEAR = {1978},
     PAGES = {662--671},
      ISSN = {},
   MRCLASS = {},
  MRNUMBER = {},
MRREVIEWER = {},
       DOI = {10.1215/ijm/1256048473},
       URL = {https://doi.org/10.1215/ijm/1256048473},
}

@article {CPPZ,
    AUTHOR = {Cheltsov, Ivan and Park, Jihun and Prokhorov, {\relax Yu}ri and
              Zaidenberg, Mikhail},
     TITLE = {Cylinders in {F}ano varieties},
   JOURNAL = {EMS Surv. Math. Sci.},
  FJOURNAL = {EMS Surveys in Mathematical Sciences},
    VOLUME = {8},
      YEAR = {2021},
    NUMBER = {1-2},
     PAGES = {39--105},
      ISSN = {2308-2151,2308-216X},
   MRCLASS = {14R25 (14E05 14E08 14E30 14J45 14J50 14R20)},
  MRNUMBER = {4307203},
       DOI = {10.4171/emss/44},
       URL = {https://doi.org/10.4171/emss/44},
}

@article {Saw24,
    AUTHOR = {Sawahara, Masatomo},
     TITLE = {Cylinders in canonical del {P}ezzo fibrations},
   JOURNAL = {Ann. Inst. Fourier (Grenoble)},
  FJOURNAL = {Universit\'e{} de Grenoble. Annales de l'Institut Fourier},
    VOLUME = {74},
      YEAR = {2024},
    NUMBER = {1},
     PAGES = {1--69},
      ISSN = {0373-0956,1777-5310},
   MRCLASS = {14R25 (14D06 14E30 14J17 14J26)},
  MRNUMBER = {4748167},
MRREVIEWER = {Jing\ Zhang},
       DOI = {10.5802/aif.3573},
       URL = {https://doi.org/10.5802/aif.3573},
}

@Misc{KKW25,
  author        = {I.-K. Kim and J. Kim and J. Won},
  howpublished  = {\href{https://arxiv.org/abs/2506.01310}{arXiv:2506.01310}},
  title         = {Rigid affine cones over singular del {P}ezzo surfaces},
  year          = {2025},
  archiveprefix = {arXiv},
  primaryclass  = {math.AG},
}

@incollection {KPZ11,
    AUTHOR = {Kishimoto, Takashi and Prokhorov, {\relax Yu}ri and Zaidenberg,
              Mikhail},
     TITLE = {Group actions on affine cones},
 BOOKTITLE = {Affine algebraic geometry},
    SERIES = {CRM Proc. Lecture Notes},
    VOLUME = {54},
     PAGES = {123--163},
 PUBLISHER = {Amer. Math. Soc., Providence, RI},
      YEAR = {2011},
      ISBN = {978-0-8218-7283-3},
   MRCLASS = {14R20 (14J26 14J50)},
  MRNUMBER = {2768637},
MRREVIEWER = {J\'er\'emy\ Blanc},
       DOI = {10.1090/crmp/054/08},
       URL = {https://doi.org/10.1090/crmp/054/08},
}

@article {KPZ13,
    AUTHOR = {Kishimoto, Takashi and Prokhorov, {\relax Yu}ri and Zaidenberg,
              Mikhail},
     TITLE = {{$\mathbb G_{\rm a}$}-actions on affine cones},
   JOURNAL = {Transform. Groups},
  FJOURNAL = {Transformation Groups},
    VOLUME = {18},
      YEAR = {2013},
    NUMBER = {4},
     PAGES = {1137--1153},
      ISSN = {1083-4362,1531-586X},
   MRCLASS = {14C20 (14J45 14R20)},
  MRNUMBER = {3127989},
MRREVIEWER = {Karol\ Palka},
       DOI = {10.1007/s00031-013-9246-5},
       URL = {https://doi.org/10.1007/s00031-013-9246-5},
}

@article {DK19a,
    AUTHOR = {Dubouloz, Adrien and Kishimoto, Takashi},
     TITLE = {Deformations of {$\mathbb{A}^1$}-cylindrical varieties},
   JOURNAL = {Math. Ann.},
  FJOURNAL = {Mathematische Annalen},
    VOLUME = {373},
      YEAR = {2019},
    NUMBER = {3-4},
     PAGES = {1135--1149},
      ISSN = {},
   MRCLASS = {},
  MRNUMBER = {3953123},
MRREVIEWER = {Bandman, Tatiana M.},
       DOI = {10.1007/s00208-018-1774-9},
       URL = {https://doi.org/10.1007/s00208-018-1774-9},
}

@article {DK19b,
    AUTHOR = {Dubouloz, Adrien and Kishimoto, Takashi},
     TITLE = {Cylindres dans les fibrations de Mori: formes du volume quintique de del {P}ezzo},
   JOURNAL = {Ann. Inst. Fourier (Grenoble)},
  FJOURNAL = {Annales de l'Institut Fourier},
    VOLUME = {69},
      YEAR = {2019},
    NUMBER = {6},
     PAGES = {2377--2393},
      ISSN = {},
   MRCLASS = {14E30, 14J30, 14J45, 14R10, 14R25},
  MRNUMBER = {4033922},
MRREVIEWER = {Tomasz Pe\lka},
       DOI = {10.5802/aif.3297},
       URL = {https://doi.org/10.5802/aif.3297},
}

@article {Saw25,
    AUTHOR = {Sawahara, Masatomo},
     TITLE = {Cylindrical ample divisors on {D}u Val del {P}ezzo surfaces},
   JOURNAL = {Forum Math.},
  FJOURNAL = {Forum Mathematicum},
    VOLUME = {37},
      YEAR = {2025},
    NUMBER = {5},
     PAGES = {1597--1619},
      ISSN = {},
   MRCLASS = {14C20, 14E05, 14J17, 14J26, 14J45, 14R25, 14J26, 14M25},
  MRNUMBER = {4912805},
MRREVIEWER = {},
       DOI = {10.1515/forum-2024-0279},
       URL = {https://doi.org/10.1515/forum-2024-0279},
}

@Misc{KL25,
  author        = {I.-K. Kim and D.-W. Lee},
  howpublished  = {\href{https://arxiv.org/abs/2507.13649}{arXiv:2507.13649}},
  title         = {{K}-stability of del {P}ezzo surfaces with a single quotient singularity},
  year          = {2025},
  archiveprefix = {arXiv},
  primaryclass  = {math.AG},
}

@Misc{KSW26,
  author        = {I.-K. Kim and M. Sawahara and J. Won},
  howpublished  = {\href{https://arxiv.org/abs/2605.10244}{arXiv:2605.10244}},
  title         = {Polarized cylinders on blow-ups of weighted projective planes},
  year          = {2026},
  archiveprefix = {arXiv},
  primaryclass  = {math.AG},
}

@article {CPW16,
    AUTHOR = {Cheltsov, Ivan and Park, Jihun and Won, Joonyeong},
     TITLE = {Affine cones over smooth cubic surfaces},
   JOURNAL = {J. Eur. Math. Soc.},
  FJOURNAL = {Journal of the European Mathematical Society (JEMS)},
      YEAR = {2016},
    NUMBER = {18},
     PAGES = {1537--1564},
      ISSN = {},
   MRCLASS = {14E15, 14J45, 14R20, 14R25, 14C20, 14E05, 14J17},
  MRNUMBER = {},
MRREVIEWER = {},
       DOI = {10.4171/JEMS/622},
       URL = {https://doi.org/10.4171/JEMS/622},
}

@article {KPZ14,
    AUTHOR = {Kishimoto, Takashi and Prokhorov, {\relax Yu}ri and Zaidenberg,
              Mikhail},
     TITLE = {Unipotent group actions on del {P}ezzo cone},
   JOURNAL = {Algebr. Geom.},
  FJOURNAL = {Algebraic Geometry},
    VOLUME = {1},
      YEAR = {2014},
    NUMBER = {1},
     PAGES = {46--56},
      ISSN = {},
   MRCLASS = {14R20, 14J45, 14J50, 14R05},
  MRNUMBER = {},
MRREVIEWER = {},
       DOI = {10.14231/AG-2014-003},
       URL = {https://doi.org/10.14231/AG-2014-003},
}

\end{document}